\documentclass[10pt]{amsart}
 
\usepackage{amsmath,amsfonts,amsthm,amssymb,amscd}
\usepackage{hyperref}
\usepackage{mathtools}

\usepackage{tikz}
\usepackage{tikz-cd}
\usetikzlibrary{positioning}
\usetikzlibrary{matrix,arrows,decorations.pathmorphing}

\usepackage{braket}
\usepackage{tabularx}
\usepackage[alphabetic]{amsrefs}
\usepackage{pgfplots}
\usepgfplotslibrary{polar}
\pgfplotsset{compat=newest}
\usepackage{graphicx}

\usepackage{subcaption}

\theoremstyle{plain}
\newtheorem{theorem}{Theorem}[section]
\newtheorem{proposition}[theorem]{Proposition}

\newtheorem{lemma}[theorem]{Lemma}

\newtheorem{corollary}[theorem]{Corollary}

\newtheorem*{claim*}{Claim}

\newtheorem*{theorem*}{Theorem}

\theoremstyle{definition}
 \newtheorem{definition}[theorem]{Definition}

\newtheorem{remark}[theorem]{Remark}
\newtheorem{example}[theorem]{Example}

\theoremstyle{plain}
\newtheorem{thmx}{Theorem}
\newtheorem{defx}[thmx]{Definition}
\newtheorem{corx}[thmx]{Corollary}
\newtheorem{propx}[thmx]{Proposition}

\newtheorem{remx}[thmx]{Remark}

\newcommand{\dist}{\mathsf{dist}}

\newcommand{\link}{\mathsf{link}}
\newcommand{\st}{\mathsf{st}}
\newcommand{\nclose}[1]{\ensuremath{\langle\!\langle#1\rangle\!\rangle}}

\newcommand{\Q}{\mathbb{Q}}  
\newcommand{\Z}{\mathbb{Z}}   %Integer numbers' symbol
\newcommand{\N}{\mathbb{N}}   %Integer numbers' symbol

\title[Topological groups with a compact open subgroup]{Topological groups with a compact open subgroup,  Relative hyperbolicity and Coherence}

\author{Shivam Arora and Eduardo Mart\'inez-Pedroza}

\email{eduardo.martinez@mun.ca}

\email{sarora17@mun.ca}

\address{Department of Mathematics and Statistics. Memorial University of Newfoundland. St. John's, NL. Canada}

 
\date{\today}

\begin{document}

\maketitle
 
\begin{abstract} 
The main objects of study in this article are pairs $(G, \mathcal{H})$ where $G$ is a topological group with a compact open subgroup, and $\mathcal{H}$ is a finite collection of open subgroups. We develop geometric techniques to study the notions of $G$ being compactly generated and compactly presented relative to $\mathcal H$. This includes topological characterizations in terms of discrete actions of $G$ on complexes, quasi-isometry invariance of certain graphs associated to the pair $(G,\mathcal H)$ when $G$ is compactly generated relative to $\mathcal H$, and  extensions of known results from the discrete case. For example, generalizing results of Osin for discrete groups, we show that in the case that $G$ is compactly presented relative to $\mathcal H$:
\begin{itemize}
    \item if $G$ is compactly generated, then each subgroup $H\in \mathcal H$ is compactly generated;
    \item if each subgroup $H\in \mathcal H$ is compactly presented, then $G$ is compactly presented.
\end{itemize}
The article also introduces an approach to relative hyperbolicity for pairs $(G, \mathcal H)$ based on Bowditch's work using discrete actions on hyperbolic fine graphs. For example, we prove that if $G$ is hyperbolic relative to $\mathcal H$ then $G$ is compactly presented relative to $\mathcal H$.

As an application of the results of the article we prove combination results for coherent topological groups with a compact open subgroup, and extend McCammond-Wise perimeter method to this general framework. 
\end{abstract}

\maketitle

\section{Introduction}

This article is part of the program of generalizing geometric techniques in the study of discrete groups to the larger class of locally compact groups. 

As a convention, all topological groups considered in the article are Hausdorff, and all group actions on CW-complexes are assumed to be cellular. Throughout the article, we work in the class of topological groups with a compact open subgroup. Note that such groups are locally compact, and by van Dantzig’s Theorem~\cite{Vd36}, all totally disconnected locally compact groups (TDLC groups) belong to this class. 
The class of TDLC groups has been a topic of interest in the last
three decades since the work of G. Willis~\cite{MR1299067},  of M. Burger and S. Mozes~\cite{MR1839489}, and P.E Caprace and N. Monod~\cite{MR2739075}. TDLC groups include profinite
groups, discrete groups, algebraic groups over non-archimedean local fields,
and automorphism groups of locally finite graphs.

Let $G$ be a topological group. The group $G$  is \emph{compactly generated} if it admits a compact generating set; it is  \emph{compactly presented} if it admits  a standard presentation $\langle S\mid R\rangle$  with $S$ a compact subset of $G$ and $R$ a set of words in $S$ of uniformly  bounded length. A topological group is said to be \emph{coherent} if every closed    compactly generated subgroup is compactly presented. 

The following remark is a consequence of standard results in the literature, see Proposition~\ref{prop:ConShortExact}. 

\begin{remx}\label{rem:CoCompactcoherent}
    Let $G$ be a topological group and $N \trianglelefteq G$ be a compact normal subgroup. Then $G$ is coherent if and only if $G/N$ is coherent.
\end{remx}

%%% DO NOT INCLUDE.  Example:  $\R\times \Z \to \R$.  The set $\{ \frac{1}{n} \times n \colon n\in \Z_+ \}$ is closed but its image is not closed.  A better example:  Consider $\R\times \R$ and a line through the origin with irrational slope. Then the image in the torus group $\R\times\R/\Z\times\Z$ is not closed.

The simplest coherent groups are compact groups.
In the discrete case, any virtually free group is also coherent, a statement that could be generalized in the framework of amalgamated free products. For topological groups, if a locally compact group splits as an amalgam or an HNN extension, then the splitting subgroup is an open subgroup, see the work of Alperin~\cite[Corollary 3]{Alperin82}. 

\begin{thmx}[Combination of Coherent Groups]\label{thm:combination}
If $G=A \ast_C B$ is a topological group that splits as an amalgamated free product of two coherent open subgroups $A$ and $B$ with compact intersection $C$, then $G$ is coherent. \end{thmx}

We give a proof of Theorem~\ref{thm:combination} at the end of the introduction. There are classical examples that illustrate the above proposition, for instance  the splitting of the discrete group $SL_2(\Z)$ as $C_4\ast_{C_2}C_6$ where $C_n$ denotes a cyclic group of order $n$.  In the non-discrete case, if $\Q_p$ denotes the field of $p$-adic numbers and $\Z_p$ the $p$-adic integers, the group $SL_2(\Q_p)$ splits as an amalgamated free product of two open subgroups isomorphic to the compact group $SL_2(\Z_p)$ along their common intersection, see for example the work of Serre~\cite{Serre}.
Let us remark that coherence of $SL_2(\Q_p)$ follows directly from a result attributed to J.Tits, see~\cite[Thm. T]{Pr82}.

Particular classes of small cancellation groups are known to be coherent by the work of McCammond and Wise~\cite{MccWise05}. We are able to use some of their methods to prove coherence in some classes of small cancellation quotients of amalgamated free products of coherent topological groups. For background on small cancellation quotients of amalgamated free products we refer the reader to the book by Lyndon and Schupp~\cite{LySc01}. We   say that an element $r$ of an amalgamated free product of groups $A\ast_CB$  satisfies the \emph{$C'(\lambda)$ small cancellation 
condition} if the smallest symmetrized subset of $A\ast_CB$ containing $r$ satisfies the condition as defined in~\cite[Ch. V. 11]{LySc01}.

\begin{thmx}\label{thmX:main}\label{thm:main}
    Let $A \ast_C B$ be a topological group that splits as an amalgamated free product of coherent   open subgroups $A$ and $B$ with compact intersection $C$. For any $r\in A\ast_CB$, there is $M>0$ such that if  $m>1$ and  $r^m$ satisfies the $C'(\lambda)$ small cancellation condition for $\lambda$ satisfying $12 \lambda M<1$, then the  quotient topological group $(A\ast_CB)/\nclose{r^m}$ is coherent. 
\end{thmx}

This result in  the case that $A$ and $B$ are free groups is a result of McCammond and Wise~\cite[Theorem 8.3]{MccWise05}, and the generalization where $A$ and $B$ are coherent discrete groups is a result of Wise and the second author~\cite[Theorem 1.8]{MARTINEZPEDROZA20112396}. The proofs of these results rely on a technique known as the perimeter method~\cite{MccWise05} which was  motivated by the well known question of Gilbert Baumslag of whether all one relator groups are coherent~\cite{MR0364463}, see the survey on coherence by Wise~\cite{Wise} and the recent work by Louder and Wilton~\cite{MR4216595}. 
 
 Theorems~\ref{thm:combination} and~\ref{thmX:main} are consequences of extensions of 
 techniques in the study of discrete groups to the framework of locally compact groups. We summarize  this work below, where the main objects of study  are pairs $(G,\mathcal{H})$ where $G$ is a topological group with a compact open subgroup and $\mathcal{H}$ is a finite collection of open subgroups. In order to avoid repetition, we introduce  the following terminology.

\begin{defx}[Proper pair]\label{def:properPair}
A pair $(G,\mathcal H)$ is called a \emph{proper pair} if 
\begin{enumerate}
    \item $G$ is a topological group with a compact open subgroup;
    \item $\mathcal{H}$ is a finite collection  of open subgroups of $G$;
    \item No pair of distinct non-compact  subgroups in $\mathcal{H}$ are conjugate in $G$.
\end{enumerate}
\end{defx}
Note that for a proper pair $(G,\mathcal H)$, we allow  $\mathcal H$ to be the empty collection.

\subsection{Compact relative generating sets.}

Let $G$ be a topological group and let $\mathcal{H}$ be a collection of subgroups.
We shall say that $G$ is \emph{compactly generated relative to $\mathcal{H}$} if there is a compact subset $S\subset G$ such that $G$ is algebraically generated by $S\cup \bigcup\mathcal{H}$. Relative compact generation can be  topologically characterized in terms of the existence of Cayley-Abels graphs in the case that $G$ contains a compact open subgroup and $\mathcal{H}$ consists of open subgroups. In order to state our result, let us recall some standard terminology. Let $X$ be a CW-complex with a $G$-action by cellular automorphisms. The $G$-action on $X$ is \emph{discrete} if pointwise stabilizers of cells are open subgroups;  in this case we say that $X$ is a \emph{discrete $G$-complex}. The $G$-complex $X$ is \emph{cocompact} if there is a compact subset $K$ of $X$ whose image under the action of $G$ covers $X$. A graph is a 1-dimensional CW-complex, and the graph is simplicial if there are no loops or multiple edges between the same vertices. 

\begin{defx}[Cayley-Abels graph]\label{defx:CayleyAbels}
Let $G$ be a topological group with a compact open subgroup, and let $\mathcal{H}$ be a finite collection of subgroups.  A \emph{Cayley-Abels graph} of  $G$ with respect to $\mathcal H$ is a connected cocompact simplicial discrete $G$-graph $\Gamma$ such that: 
\begin{enumerate}
    \item edge $G$-stabilizers are compact,
    \item vertex $G$-stabilizers are either compact or conjugates of subgroups in $\mathcal{H}$,
    \item every $H\in \mathcal{H}$ is the $G$-stabilizer of a vertex of $\Gamma$, and
    \item any pair of vertices of $\Gamma$ with the same $G$-stabilizer $H\in\mathcal H$ are in the same $G$-orbit if $H$ is non-compact.
\end{enumerate}
A consequence of these   conditions is:
\begin{enumerate}
    \item[(5)] If $\Gamma$ contains at least one edge, then a vertex has finite degree if and only if its $G$-stabilizer is compact. 
\end{enumerate}
\end{defx}

The notion of Cayley-Abels graph  generalizes the notion of Cayley graph of a finitely generated group. Other examples are discussed in detail in Section~\ref{sec03:CayleyAbels}. Let us mention that the Bass-Serre tree of a finite graph of groups with finite edge groups is a Cayley-Abels graph for its fundamental group with respect to the vertex groups; Farb's coned-off Cayley graph of a relatively hyperbolic group~\cite{Farb} is  a Cayley-Abels graph for the group with respect to its peripheral structure; and  
Kr\"{o}n and M\"{o}ller's rough Cayley graph of a compactly generated TDLC group  is a Cayley-Abels graph for the group with respect to the empty collection~\cite{BM08}. 

\begin{thmx}[Topological characterization of relative  compact  generation] \label{thmX:TopCharGraph} 
Let $(G,\mathcal H)$ be a proper pair. The following statements are equivalent:
\begin{enumerate}
    \item $G$ is compactly generated relative to $\mathcal{H}$.
    \item There exists a Cayley-Abels graph of $G$ with respect to $\mathcal{H}$.
\end{enumerate}
 \end{thmx}

In the case that $G$ is a discrete group  and $\mathcal H$ is empty, the above result is the well-known fact that a group is finitely generated if and only if it acts properly and cocompactly on a connected graph. In the case that $G$ is discrete and $\mathcal{H}$ is a finite collection of subgroups, the result appears implicit in the work of Hruska on relatively hyperbolic groups~\cite{Hru} where the  resulting graphs are called \emph{coned-off Cayley graphs}.   
In the case that $G$ is a  TDLC group  and $\mathcal{H}$ is empty, this is a result of Kr\"{o}n and M\"{o}ller's from~\cite{BM08} where they show that  the resulting graphs can be assumed to be vertex transitive and call them \emph{rough Cayley graphs}; these graphs are also known as \emph{Cayley-Abels graphs} after related work of Herbert Abels~\cite{MR344375}. In the case that $G$ has a compact open subgroup and $\mathcal{H}$ is empty, the theorem is a result of Cornulier and de la  Harpe~\cite[Proposition 2.E.9]{CoHa16}.

 The next result is on the quasi-isometry invariance of  {Cayley-Abels graphs} of a proper pair $(G,\mathcal H)$. These graphs are not necessarily locally finite.  The  notion of fineness, introduced by Bowditch~\cite{Bo12} and defined below, is a generalization of local finiteness of graphs which turns out to be relevant in our context, see Theorem~\ref{thmX:CAcomplexchar} and the comments after its  statement. %Fineness is shared by all  Cayley-Abels graphs of a proper pair if one of them has it.
 
 \begin{defx}[Bowditch fineness]\cite{Bo12} A simplicial graph $\Gamma$ is fine if for any pair of vertices $u,v$ and any integer $n$, there are finitely many embedded paths of length $n$ from $u$ to $v$.
\end{defx}

 \begin{thmx}[Quasi-isometry invariance]\label{thmX:uniqueCAgraph}
Let $(G,\mathcal H)$ be a proper pair. If $\Gamma$ and $\Delta$ are Cayley-Abels graphs of $G$ with respect to $\mathcal{H}$, then
\begin{enumerate}
    \item $\Gamma$ and  $\Delta$ are quasi-isometric, and
    \item   $\Gamma$ is fine if and only if $\Delta$ is fine.
\end{enumerate}
\end{thmx}

This is a new result even in the discrete case. For pairs $(G,\mathcal H)$ with $G$ a TDLC group $G$ and  $\mathcal H$ the   empty collection, the statement is a result of Kr\"{o}n and M\"{o}ller's~\cite{BM08}. The case for pairs $(G,\mathcal H)$ where $G$ is a finitely generated  discrete group and $\mathcal H$ is a malnormal collection is a consequence of~\cite[Prop. 5.6]{SamLouisEdu}.
The quasi-isometry invariance of  Cayley-Abels graphs  allow us to define geometric invariants for pairs $(G,\mathcal{H})$ where $G$ is a topological group with a compact open subgroup and $\mathcal{H}$ is a finite collection of open subgroups.
For example, hyperbolicity in the class of topological groups with a compact open subgroup can be defined as the groups that admit a hyperbolic Cayley-Abels graph (with respect to the empty collection) an approach considered in~\cite{ACCMP}. The  approach to hyperbolic groups   for locally compact groups  developed by Caprace, Cornulier, Monod and Tessera~\cite{MR3420526}   when restricted to locally compact groups with a compact open subgroup is equivalent. We use Theorem~\ref{thmX:uniqueCAgraph} to  start the development of  a theory of relatively hyperbolic groups for topological groups with a compact open subgroup, see subsection~\ref{intro-relhyp}.

\subsection{Compact relative presentations}

Let $G$ be a topological group and let $\mathcal H=\{H_1,\ldots, H_n\}$ be  a finite collection  of open subgroups. We say that $G$ is \emph{compactly presented relative to $\mathcal{H}$} if there is a short exact sequence 
\[ 1 \to \nclose{R} \to \pi_1(\mathcal{G},\Lambda) \xrightarrow{\phi} G \to 1      \]
where $\pi_1(\mathcal{G},\Lambda)$ is the fundamental group of a finite graph of topological groups $(\mathcal{G},\Lambda)$
endowed with the topology induced by the vertex groups (see Definition~\ref{def:TopGraphGroups} and Proposition~\ref{prop:GTopology})  with the following properties.
\begin{itemize}
\item There are vertices $\{v_1,v_2, \cdots v_n\}$ of $\Lambda$ and isomorphisms of topological groups  $\phi_i\colon \mathcal{G}_{v_i}\to H_i$ such that \[\begin{tikzcd}[column sep=small]
\mathcal{G}_{v_i} \arrow[hookrightarrow]{r}{i_{v_i}}  \arrow{d}{\phi_i} 
&[1em]  \pi_1(\mathcal{G}, \Lambda) \arrow{d}{\phi} \\
H_i \arrow[hookrightarrow]{r} &  G
\end{tikzcd}\] is a commutative diagram up to an inner automorphism of $G$.
\item For every $v \in V(\Lambda)$, the map $\phi \circ i_v$ is injective.

\item For each edge $e \in E(\Lambda)$ and each vertex $v \neq v_i$ in $V(\Lambda)$, the edge group $\mathcal{G}_e$ and the vertex group $\mathcal{G}_{v}$ are compact topological groups.

\item $\phi$ is   a continuous open   epimorphism whose restriction to each vertex group of $(\mathcal{G},\Lambda)$ is injective.

\item $\nclose{R}$ is a discrete normal subgroup generated by a finite subset $R$ of $\pi_1(\mathcal{G},\Lambda)$.
\end{itemize}
The pair $\langle  (\mathcal{G},\Lambda,\phi) \mid R \rangle$ is called a \emph{compact generalized presentation} of $G$ with respect to $\mathcal{H}$.   

In the case that $G$ is compactly presented with respect to the empty collection,  we say that $G$ is \emph{compactly presented}. This is  equivalent to the definition stated at the beginning of the introduction, see Corollary~\ref{prop:CompactPres}. In the case that $G$ is discrete and $\mathcal{H}$ is empty, the definition is equivalent to $G$ being the quotient of a virtually free group of finite rank by a normal subgroup generated by a finite number of elements. In the case that $G$ is a discrete group and $\mathcal{H}$ is not empty, it is an observation that our definition of $G$ being finitely presented relative to $\mathcal{H}$ coincides with the approach by Osin~\cite{Os06}. We previously mentioned that in the case that $G$ is a TDLC group and $\mathcal{H}$ is empty, this approach was used by Castellano~\cites{CaIWei16,castellano_2020}.

\begin{thmx}[Topological Characterization of Relative compact presentation]\label{thmX:CAcomplexchar}
Let $G$ be a topological group with a finite collection $\mathcal{H}$ of open subgroups. The following statements are equivalent:
\begin{enumerate}
    \item $G$ is compactly presented with respect to $\mathcal{H}$.
    \item There exists a    Cayley-Abels graph $\Gamma$ of $G$ with respect to $\mathcal{H}$ which is  the 1-skeleton of a simply-connected cocompact  discrete $G$-complex.
    \end{enumerate}
\end{thmx}

In the discrete case, that (2) implies (1)  can be found in~\cite[Prop. 4.16]{SamLouisEdu}. Note that for a discrete group $G$ with an empty collection a stronger version of Theorem~\ref{thmX:CAcomplexchar} holds in the sense that  the second item can be expressed with a universal quantifier. In order to obtain this type of equivalence in our context we need to impose an additional hypothesis:
\begin{corx} \label{thmX:TopCharComplex} 
Let $(G,\mathcal H)$ be a proper pair. Suppose there is a fine Cayley-Abels graph  of $G$ with respect to $\mathcal{H}$. The following statements are equivalent:
\begin{enumerate}
    \item $G$ is compactly presented with respect to $\mathcal{H}$.
    \item Any  Cayley-Abels graph of $G$ with respect to $\mathcal H$ is the 1-skeleton of a simply-connected cocompact  discrete $G$-complex.
\end{enumerate}
\end{corx}

We are not aware whether the assumption on fineness of all Cayley-Abels graphs is necessary. There is a well known relationship between fineness and isoperimetric functions which was first made explicit by Groves and Manning~\cite[Proposition 2.50, Question 2.51]{GrMa09}. This relation has also been studied in~\cites{MP15, SamLouisEdu}.  The 2-complexes given by Corollary~\ref{thmX:TopCharComplex}, which are  not necessarily locally finite, satisfy the hypothesis of the following result.

\begin{propx}\cite{MP15} \cite{SamLouisEdu}\label{propx:Deltafine}
Let $X$ be a cocompact simply-connected $G$-complex. Suppose that each edge of $X$ is attached to finitely many 2-cells. Then the 1-skeleton of $X$ is a fine graph if and only if the combinatorial Dehn function of $X$ takes only finite values. 
\end{propx}

Corollary~\ref{thmX:TopCharComplex}  in the case that  $G$ is a TDLC group and $\mathcal H$ is the empty collection is a result of Castellano and Cook~\cite[Proposition 3.4]{Cook}, and for generally locally compact groups with empty $\mathcal H$, there is a version by Cornulier and de la Harpe, see~\cite[Corollary 8.A.9]{CoHa16}. 

In the context of topological groups admitting compact presentations, Theorem~\ref{thmX:FinRelPresHisCPimpliesGisCP} below is our main result.  In the discrete case, it is a result of Osin~\cite[Theorem 1.1 and Theorem 2.40]{Os06}. Let us remark that our argument uses different techniques than the ones in the cited article.  

  \begin{thmx} \label{thmX:FinRelPresHisCPimpliesGisCP}
   Let $(G,\mathcal H)$ be a proper pair. Suppose that $G$ is compactly presented relative to $\mathcal{H}$.
\begin{enumerate}
\item If each $H \in \mathcal{H}$ is compactly presented then $G$ is compactly presented.
\item If $G$ is compactly generated, then each $H \in \mathcal{H}$ is compactly generated. 
\end{enumerate}
\end{thmx}

\subsection{Relative hyperbolicity}\label{intro-relhyp}

 Let $(G, \mathcal{H})$ be a proper pair. The topological group $G$ is   \emph{   hyperbolic relative} to $\mathcal{H}$ if there exists a  Cayley-Abels graph of $G$ with respect to $\mathcal H$ that is fine and hyperbolic. 
  
 \begin{remx}  Let $(G, \mathcal{H})$ be a proper pair.
 If $G$ is hyperbolic relative to    $\mathcal H$   then: 
 \begin{enumerate}
\item  $G$ is compactly generated relative to $\mathcal H$ by Theorem~\ref{thmX:TopCharGraph};  and 
\item   Cayley-Abels graphs of $G$ with respect to $\mathcal{H}$ are fine and hyperbolic by  Theorem~\ref{thmX:uniqueCAgraph}.
\end{enumerate}
 \end{remx}

Relatively hyperbolic groups were introduced by Gromov~\cite{Gromov}. Our definition of relative hyperbolicity when restricted to discrete groups coincides with an approach by Bowditch~\cite{Bo12} which is also equivalent to the approach by Osin~\cite{Os06}. In the case of that a topological group $G$ with a compact open subgroup is hyperbolic relative to  the empty collection, the group $G$ is hyperbolic in the sense of Baumgartner, M\"oller and Willis~\cite{BaMo12}, and in the sense of Cornulier and Tessera~\cite{CorTess11}.

\begin{thmx}\label{thmX:RelHypimpliesFinitePres}
Let $(G,\mathcal H)$ be a proper pair. Suppose $G$ is hyperbolic relative to  $\mathcal{H}$. Then $G$ is compactly presented relative to $\mathcal{H}$.
\end{thmx}

Putting together Theorems~\ref{thmX:FinRelPresHisCPimpliesGisCP} and~\ref{thmX:RelHypimpliesFinitePres} we obtain:

\begin{corx}\label{cor:N}
   Let $(G,\mathcal H)$ be a proper pair. Suppose $G$ is hyperbolic relative to  $\mathcal{H}$.
\begin{enumerate}
    \item If each $H \in \mathcal{H}$ is compactly presented then $G$ is compactly presented. 
    \item If $G$ is compactly generated, then each $H \in \mathcal{H}$ is compactly generated.
\end{enumerate}    
\end{corx}

\subsection{Proof of Theorem~\ref{thm:combination} and $\mathcal J$-coherence}
Let us conclude the introduction by proving   Theorem~\ref{thm:combination}  using the results that have been stated. We   prove the slightly more general result Theorem~\ref{thmx:conclusion1}. 

Let $G$ be a topological group and let $\mathcal J$ be a class of   subgroups. We say that $G$ is \emph{$\mathcal J$-coherent} if every compactly generated subgroup of $G$ that belongs to $\mathcal J$ is compactly presented. In particular, being coherent means being $\mathcal J$-coherent for $\mathcal J$ the class of all closed subgroups. For $A \leq G$, we say that $A$ is $\mathcal J$-coherent if $A$ is $\mathcal K$-coherent for $\mathcal K = \{Q\in \mathcal J \mid Q\leq A\}$. Classes of subgroups  satisfying the hypothesis of the theorem below include discrete subgroups, open subgroups and closed subgroups.

\begin{thmx}\label{thmx:conclusion1}
Let $G=A \ast_C B$ be a topological group that splits as an amalgamated free product of two open subgroups $A$ and $B$ with compact intersection $C$. Let $\mathcal J$ be a collection of closed subgroups of $G$ with the following properties:
\begin{enumerate}
    \item  $\mathcal J$ is closed under conjugation and finite intersections,  and
   \item $A$ and $B$ belong to $\mathcal{J}$.
\end{enumerate}
If $A$ and $B$ are $\mathcal J$-coherent, then $G$ is $\mathcal J$-coherent. 
\end{thmx}
\begin{proof}
Let $T$ be the Bass-Serre tree of the splitting $G=A\ast_C B$. Observe that the $G$-action on $T$ is discrete since $A$, $B$ and $C$ are open subgroups of $G$. Moreover, edge stabilizers in $T$ are compact since $C$ is compact.

Let $Q$ be a compactly generated   subgroup of $G$ that belongs to $\mathcal J$. First suppose that  $Q$ fixes a vertex of $T$. In this case $Q$ is conjugate to a subgroup of $A$ or $B$. Since $A$ and $B$ are $\mathcal J$-coherent,  we have that $Q$ is compactly presented. 

Suppose that $Q$ does not fix a vertex of $T$ and let $U=Q\cap C$. Since $Q$ is closed, we have that $U$ is a compact open subgroup of $Q$. Since $Q$ is compactly generated and $U$ is open in $Q$, there is a finite subset $S\subset Q$ such that $S\cup U$ generates $Q$. Let $e$ be an edge of $T$ which is stabilized by $U$, let $D$ be the minimal connected subgraph of $T$ containing $e$ and $g.e$ for all $g\in S$. Since $S$ is finite, $D$ is a finite subtree. Since $S\cup U$ generates $Q$, it follows that   $\Delta=\bigcup_{g\in Q} gD$ is a 
subtree of $T$. Moreover,  $Q$ acts on $\Delta$ discretely, cocompactly and with compact edge stabilizers. Let $\mathcal{H}$ be a collection of representatives of conjugacy classes of vertex $Q$-stabilizers of $\Delta$. Since $\mathcal{H}$ is finite,    $(Q,\mathcal H)$ is a proper pair and therefore $\Delta$ is a  Cayley-Abels graph of $Q$ with respect to $\mathcal{H}$. Since trees are hyperbolic and fine, it follows that $Q$ is hyperbolic relative to $\mathcal{H}$. Since $Q$ is compactly generated, the second part of Corollary~\ref{cor:N} implies that each $H\in \mathcal{H}$ is compactly generated. By definition,   each $H\in\mathcal{H}$ is the intersection of $Q$ with a conjugate of $A$ or $B$, in particular, $\mathcal H \subset \mathcal J$. Therefore, up to conjugation in $G$, each $H\in \mathcal{H}$ is a compactly generated subgroup of $A$ or $B$ that belongs to $\mathcal J$. Since $A$ and $B$ are  $\mathcal J$-coherent, each subgroup in $\mathcal H$ is compactly presented. Then the first part of Corollary~\ref{cor:N} implies that $Q$ is compactly presented. \end{proof}

 Let us remark that in Section~\ref{sec:main} we  prove  an analog of Theorem~\ref{thmX:main} for $\mathcal J$-coherence, see Theorem~\ref{thm:mainCoherence2}.

 \subsection*{Organization} The rest of the article is organized into nine   sections. Each section contains the proof of one the theorems stated in the introduction, except  Section~\ref{sec:CompactGenGraph} which contains some preliminaries; Section~\ref{sec03:CayleyAbels} which introduces an alternative definition of Cayley-Abels graph and discusses a proof of the equivalence with the definition in the introduction;  and  Section~\ref{sec:appendix} that contains a technical result about Bowditch's fineness that is used in Section~\ref{sec:QI-Fineness}.  The mapping of results in the introduction with the main body of the article is described in the table:
 
 \begin{center}
\begin{tabular}{ |l c l c| } 
 \hline
 Definition~\ref{defx:CayleyAbels} & Section~\ref{sec03:CayleyAbels} &  Definition~\ref{def:RelativeGraphs} and Theorem~\ref{thm:Cayley-Abels} &   \\
 \hline
 Theorem~\ref{thmX:TopCharGraph} & Section~\ref{sec04:RelGen}    & 
 Theorem~\ref{prop:TopCharGraph}& \\
 \hline
  Theorem~\ref{thmX:uniqueCAgraph}& Section~\ref{sec:QI-Fineness}&
 Corollary~\ref{cor:QI-Fineness} & \\ \hline
 Theorem~\ref{thmX:CAcomplexchar} & Section~\ref{sec:TopCharComplex} & Theorem~\ref{prop:TopCharComplex} & \\ 
Corollary~\ref{thmX:TopCharComplex} & & Corollary~\ref{cor:TopcharCWcomplex} & \\  
 \hline
 Theorem~\ref{thmX:FinRelPresHisCPimpliesGisCP} & Section~\ref{Sec:Osin2} & Theorem~\ref{prop:FiniteRelPresGisCGimpliesHisCG}  &\\
 \hline
 
 Theorem~\ref{thmX:RelHypimpliesFinitePres} & Section~\ref{sec:relhp} & Theorem~\ref{prop:RelHypimpliesFinitePres} & \\
Proposition~\ref{propx:Deltafine} & &
Proposition~\ref{thm:Deltafine} & \\ 
 \hline
 Theorem~\ref{thmX:main} & Section~\ref{sec:main} & Theorem~\ref{thm:mainCoherence2} & \\
 \hline
\end{tabular}
\end{center}

 \subsection*{Acknowledgements}
Most results in the article are based on  work in the PhD thesis of the first author. We thank the referees of the article for the tremendous feedback, suggestions and corrections. Both authors thank Ilaria Castellano, Sam Hughes, Mikhail  Kotchetov, Luis Jorge S\'anchez Salda\~na,  and Nicolas Touikan  for a number of helpful discussions on the topics of this work and for comments and corrections on a preliminary version of the manuscript.  The second author acknowledges funding by the Natural Sciences and Engineering Research Council of Canada, NSERC.

\section{Graphs of topological groups and Compact generating graphs}
\label{sec:CompactGenGraph}

This is a preliminary section on the notion of graph of topological groups, the canonical topology on their fundamental groups, and  the notion of compact generating graph, Definition~\ref{def:GenGraph},  which are used in the rest of the article. 

For background on  graphs of groups we refer the reader to Serre's book on trees~\cite{Serre}. 
Let us recall that a graph $\Lambda$ in the sense of Serre consists of a vertex set $V=V(\Lambda)$, 
an edge set $E=E(\Lambda)$, a map $E\to E$, $e\mapsto \overline  e$ such that $\overline{\overline e }=e$, 
and maps $o\colon E\to V$ and $t\colon E\to V$ such that $o(e)=t(\overline e)$. Graphs in the sense of Serre have   natural realizations as 1-dimensional CW-complexes; we work with these realizations throughout the article. 

\begin{definition}[Graph of topological groups]\label{def:TopGraphGroups}
Let $\Lambda= (V,E)$ be a connected graph. A \emph{graph of topological groups} $(\mathcal{G}, \Lambda)$ based on the graph $\Lambda$ consists of 
\begin{enumerate}
\item a topological group $\mathcal{G}_v$ for every vertex $v \in V$.
\item a topological group $\mathcal{G}_e$ for every edge $e \in E$ with $\mathcal{G}_e = \mathcal{G}_{\bar{e}}$
\item an open continuous monomorphism $\eta_e\colon \mathcal{G}_e \to \mathcal{G}_{t(e)}$ for every edge $e \in  E$. 
\end{enumerate}

Denote by $\pi_1(\mathcal{G},\Lambda, a)$ the fundamental group of the graph of  groups $(\mathcal{G},\Lambda)$, where $a$ is a vertex of $\Lambda$, see~ \cite{BASS1993}. There are 
canonical monomorphisms \[i_v\colon \mathcal{G}_v \hookrightarrow \pi_1(\mathcal{G},\Lambda, a)\] up to conjugation. 
Since the group is independent of the choice of $a$, we  simply denote it by $\pi_1(\mathcal{G},\Lambda)$. 
\end{definition}

The fundamental group of a topological graph of groups $(\mathcal{G},\Lambda)$ admits a canonical topology as described in the following proposition. Recall that an \emph{embedding} of topological spaces is a continuous map $f \colon X \to Y$ that is a homeomorphism onto its image. An embedding is   \emph{open} if $f$ is an open map. 
Specifically an open embedding is an injective, open, and continuous map.  

\begin{proposition}\cite[Propositions 8.B.9 and 8.B.10]{CoHa16} \label{prop:GTopology}
Let $(\mathcal{G}, \Lambda)$ be a finite  graph of topological groups. There exists a unique topology on $\pi_1(\mathcal{G},\Lambda)$ such that $\mathcal{G}_v \hookrightarrow \pi_1(\mathcal{G},\Lambda)$ is an open topological embedding for each vertex $v$. Moreover if vertex and edge groups are locally compact then $\pi_1(\mathcal{G},\Lambda)$ is locally compact.
\end{proposition}

From here on, all graphs of topological groups $(\mathcal{G}, \Lambda)$ are assumed to satisfy that $\Lambda$ is a finite connected graph and we always consider $\pi_1(\mathcal{G}, \Lambda)$ as a topological group with the topology provided by Proposition~\ref{prop:GTopology}.

\begin{definition}[Compact generating graph] \label{def:GenGraph}
Let $G$ be a topological group with a finite collection $\mathcal{H} = \{H_1,H_2,\cdots H_n\}$ of  subgroups. A \emph{compact generating graph of $G$ relative to} $\mathcal{H}$ is a triple $(\mathcal{G},\Lambda,\phi)$ where $(\mathcal{G},\Lambda)$ is a finite graph of topological groups, $\phi\colon \pi_1(\mathcal{G},\Lambda)  \to G$ is a continuous open surjective homomorphism, and the following properties hold:
\begin{enumerate}
\item \label{eq:comDiag} There are vertices $\{v_1,v_2, \cdots v_n\}$ of $\Lambda$ and isomorphisms of topological groups  $\phi_i\colon \mathcal{G}_{v_i}\to H_i$ such that 
\[\begin{tikzcd}[column sep=small]
\mathcal{G}_{v_i} \arrow[hookrightarrow]{r}{i_{v_i}}  \arrow{d}{\phi_i} 
&[1em]  \pi_1(\mathcal{G}, \Lambda) \arrow{d}{\phi} \\
H_i \arrow[hookrightarrow]{r} &  G
\end{tikzcd}\]
is a commutative diagram up to an inner automorphism of $G$.
\item For every $v \in V(\Lambda)$, the map $\phi \circ i_v$ is injective.

\item For each edge $e \in E(\Lambda)$ and each vertex $v \neq v_i$ in $V(\Lambda)$, the edge group $\mathcal{G}_e$ and the vertex group $\mathcal{G}_{v}$ are compact topological groups.
\end{enumerate}
\end{definition}

\begin{example}\label{exemp:fgGroup}
Let $G$ be a finitely generated group with a finite generating set $X = \{x_1,x_2, \ldots ,x_t\}$. Consider a compact generating graph $(\mathcal{G},\Lambda,\phi)$  of $G$ relative to the empty collection with $\Lambda$ given by the Figure~\eqref{fig:a}, %%\ref{fig:animals}(A)  
where all vertex and edge groups are  trivial. Note that $\pi_1(\mathcal{G},\Lambda)$ is isomorphic to the free group $F(X)$ generated over $X$ and $\phi \colon F(X) \to G$ is the natural quotient map.
\end{example}

\begin{example}
\label{exemp:RelfgGroup}
Let $G$ be a discrete group finitely generated with respect to a finite collection of groups $\mathcal{H}= \{H_1, H_2\cdots H_n\}$ in the sense of Osin~\cite[Definition 2.1]{Os06}. In particular, there exists a finite set $X \subseteq G$ such that $(\bigcup\limits_{i=1}^{n}H_i) \cup X$ generates $G$. A   compact generating graph $(\mathcal{G},\Lambda,\phi)$  of $G$ relative to $\mathcal{H}$ is given by the graph $\Lambda$ described in Figure~\eqref{fig:b}  %%\ref{fig:animals}(B)  
where  edge groups and the central vertex group are trivial, and the other vertex groups are  given by $H \in \mathcal{H}$. The homomorphism is the natural quotient map $ \phi \colon \pi_1(\mathcal{G},\Lambda) \to G$.
\end{example}

\begin{example}
\label{exemp:TDLCfgGroup}
Let $G$ be a topological group with a compact open subgroup $U$ and a finite set $X$ such that $X \cup U$ generates $G$.  A compact generating graph $(\mathcal{G},\Lambda,\phi)$  of $G$ relative to the empty collection is obtained by taking the graph   $\Lambda$ described in Figure~\eqref{fig:a}  %%\ref{fig:animals}(A)   
with vertex group $U$ and edge groups given by $U \cap U^x$ for $x \in X$. 
\end{example}

\begin{figure} 
    \centering
    \begin{subfigure}[b]{0.3\textwidth}
        \begin{tikzpicture}
\def\bouquetN{5}
\pgfmathsetmacro{\MaxAng}{360/\bouquetN}
\draw[black,thick] (0,0)  
to[out={\MaxAng*2-\MaxAng/4},in={\MaxAng*2-90}] ({\MaxAng*2+2}:1);
\draw[black,thick] (0,0) 
to[out={\MaxAng*2+\MaxAng/4},in={\MaxAng*2+90}] ({\MaxAng*2}:1)node[left= 1pt]{$x_{t}$};
\draw[black,thick] (0,0)  
to[out={\MaxAng*3-\MaxAng/4},in={\MaxAng*3-90}] ({\MaxAng*3+2}:1);
\draw[black,thick] (0,0) 
to[out={\MaxAng*3+\MaxAng/4},in={\MaxAng*3+90}] ({\MaxAng*3}:1)node[left= .5pt]{$x_{1}$};

\draw[black,thick] (0,0)  
to[out={\MaxAng*4-\MaxAng/4},in={\MaxAng*4-90}] ({\MaxAng*4+2}:1);
\draw[black,thick] (0,0) 
to[out={\MaxAng*4+\MaxAng/4},in={\MaxAng*4+90}] ({\MaxAng*4}:1)node[below = .5pt]{$x_{2}$};

\draw[black,thick] (0,0)  
to[out={\MaxAng*5-\MaxAng/4},in={\MaxAng*5-90}] ({\MaxAng*5+2}:1);
\draw[black,thick] (0,0) 
to[out={\MaxAng*5+\MaxAng/4},in={\MaxAng*5+90}] ({\MaxAng*5}:1)node[below right= .5pt]{$x_{3}$};

\foreach \X in {-1,0,1}
{\draw[black,fill=black]  ({\MaxAng+\X*\MaxAng/4}:0.8)  circle (0.5pt); }
\end{tikzpicture}
        \caption{}
        \label{fig:a}
    \end{subfigure}
    ~ %add desired spacing between images, e. g. ~, \quad, \qquad, \hfill etc. 
    \begin{subfigure}[b]{0.3\textwidth}
        \begin{tikzpicture}
\def\bouquetN{5}
\pgfmathsetmacro{\MaxAng}{360/\bouquetN}
\draw[black,thick] (0,0)  
to[out={\MaxAng*2-\MaxAng/4},in={\MaxAng*2-90}] ({\MaxAng*2+2}:1);
\draw[black,thick] (0,0) 
to[out={\MaxAng*2+\MaxAng/4},in={\MaxAng*2+90}] ({\MaxAng*2}:1)node[left= 1pt]{$x_{t}$};

\draw[black,thick] (0,0)  
to[out={\MaxAng*3-\MaxAng/4},in={\MaxAng*3-90}] ({\MaxAng*3+2}:1);
\draw[black,thick] (0,0) 
to[out={\MaxAng*3+\MaxAng/4},in={\MaxAng*3+90}] ({\MaxAng*3}:1)node[below =1pt]{$x_{3}$};

\draw[black,thick] (0,0)  
to[out={\MaxAng*4-\MaxAng/4},in={\MaxAng*4-90}] ({\MaxAng*4+2}:1);
\draw[black,thick] (0,0) 
to[out={\MaxAng*4+\MaxAng/4},in={\MaxAng*4+90}] ({\MaxAng*4}:1)node[below = .5pt]{$x_{2}$};

\draw[black,thick] (0,0)  
to[out={\MaxAng*5-\MaxAng/4},in={\MaxAng*5-90}] ({\MaxAng*5+2}:1);
\draw[black,thick] (0,0) 
to[out={\MaxAng*5+\MaxAng/4},in={\MaxAng*5+90}] ({\MaxAng*5}:1)node[below right= .5pt]{$x_{1}$};

\draw[black,thick] (0,0) -- (.1,1) node[above right = .1pt]{$H_2$};
\draw[black,thick] (.1,1) circle (0.5pt);
\draw[black,thick] (0,0) -- (-.4,1) node[above]{$H_1$};
\draw[black,thick] (-.4,1) circle (0.5pt);
\draw[black,thick] (0,0) -- (1,0.5) node[right]{$H_n$};
\draw[black,thick] (1,0.5) circle (0.5pt);
\foreach \X in {5.5,6.5,7.5}
{\draw[black,fill=black]  ({\MaxAng+\X*\MaxAng/4}:1.2)  circle (0.5pt); }
\foreach \X in {-0,-1,-2}
{\draw[black,fill=black]  ({\MaxAng+\X*\MaxAng/4}:0.8)  circle (0.5pt); }

\end{tikzpicture}
        \caption{}
        \label{fig:b}
    \end{subfigure}
    ~ %add desired spacing between images, e. g. ~, \quad, \qquad, \hfill etc. 
    \caption{Compact generating graphs}\label{fig:animals}
\end{figure}
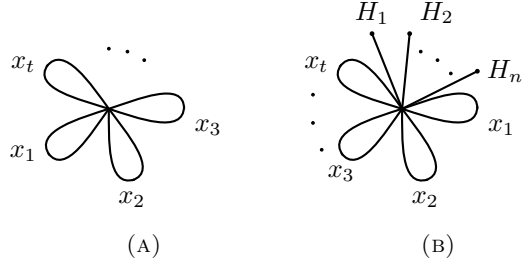 

\begin{example}\label{examp:CompGenGraphAmalgamation}
Let $A \ast_C B$ be an amalgamated free product of topological  groups $A$ and $B$ over a common compact open subgroup with the  topology such that the  inclusions $A\hookrightarrow A\ast_C B$ and 
$B\hookrightarrow A\ast_C B$ are open topological embeddings, see   Proposition~\ref{prop:GTopology}.    Let $N$  be a normal subgroup of $A\ast_C B$ and $G =(A \ast_C B) /N$ the quotient topological group. If the compositions  $A\hookrightarrow A\ast_C B \to G$ and $B\hookrightarrow A\ast_C B \to G$ are injective and  $\phi\colon A\ast_CB\to G$ is the quotient map  then, abusing notation, $(A\ast_C B, \phi)$ is a compact generating graph for $G$ with respect to the collection of open subgroups $\{A,B\}$.
\end{example}

The definition of compact generating graph implies the following observations.

\begin{remark} 
\label{rem:GenGraph}
Let $G$ be a topological group admitting a compact generating graph $(\mathcal{G}, \Lambda, \phi)$ relative to a finite collection $\mathcal{H}$ of  subgroups.
\begin{enumerate}
    \item If $G$ has no compact open subgroup, then $\mathcal{H}=\{G\}$ and $\Lambda$ has no edges.

 \item Each $H\in \mathcal{H}$ is an open subgroup of $G$.

\item If no pair of distinct subgroups in $\mathcal{H}$ are conjugate in $G$, then $(G,\mathcal H)$ is a proper pair.
\end{enumerate}
\end{remark}

\section{Cayley-Abels Graphs}
\label{sec03:CayleyAbels}

In this section, we introduce an  alternative definition of Cayley-Abels graph, Definition~\ref{def:RelativeGraphs}, and prove that it is equivalent to Definition~\ref{defx:CayleyAbels} in the introduction for proper pairs, see Theorem~\ref{thm:Cayley-Abels}.

The following definition relies on the notion of Bass-Serre tree associated to a graph of groups~\cite{Serre}. 
For the rest of the article when considering a Bass-Serre tree, we mean the realization as a CW-complex of dimension at most one as described in the book by Serre~\cite[Pg. 14]{Serre}.  

\begin{definition}[Cayley-Abels graph induced by a compact generating graph]\label{def:RelativeGraphs}
Let  $(\mathcal{G},\Lambda,\phi)$ be a compact generating graph of a topological group $G$ relative to a finite collection $\mathcal{H}$ of open subgroups, and let $\mathcal{T}$ be the corresponding Bass-Serre tree. The \emph{Cayley-Abels graph  $\Gamma(\mathcal{G},\Lambda,\phi)$ of $G$ with respect to $\mathcal{H}$ induced by $ (\mathcal{G},\Lambda,\phi)$} is defined as the  $G$-graph $\mathcal{T}/\ker(\phi)$. A $G$-graph  isomorphic to   $\Gamma(\mathcal{G},\Lambda,\phi)$ for some compact generating graph $(\mathcal{G},\Lambda,\phi)$ of $G$ with respect to $\mathcal{H}$  is called a \emph{Cayley-Abels graph of $G$ with respect to $\mathcal{H}$}. 
\end{definition}

\begin{example}[Cayley graphs of finitely generated groups]  The Cayley graph of a discrete group $G$ with respect to a finite generating set $X$ is defined as a graph $\Gamma$  with vertex set $V(\Gamma)= G$ and edge set $E(\Gamma)= \{\{g,gx\} \mid g \in G, x \in X\}$. Observe that 
the  Cayley-Abels graph of $G$ induced by  the  generating graph of $G$ described in Example~\ref{exemp:fgGroup}     is $G$-isomorphic to the Cayley graph with respect to the generating set $X$.
\end{example}

\begin{example}[Farb's Coned-off Cayley Graph~\cite{Farb}]
Let $G$ be a group, let $\mathcal{H} = \{H_1,H_2 \cdots, H_n\}$ be a finite collection of subgroups, and let $X\subset G$ be a relative finite generating set of $G$ with respect to $\mathcal{H}$. The \emph{Coned-off Cayley graph $\hat\Gamma(G,\mathcal{H}, X)$} of $G$ with respect  $\mathcal{H}$ is  the graph $\hat\Gamma$ with the vertex set $V(\hat\Gamma) =  (\bigcup\limits_{i= 1}^{n} G/H_i) \cup G$ and edge set $E(\hat\Gamma) = \{ \{g, gx\} \mid g \in G, x \in X\} \cup \{\{k, gH\} \mid  g \in G, H \in \mathcal{H}, k \in gH\}$. 
The Cayley-Abels graph of $G$ with respect to  $\mathcal{H}$ induced by the generating graph described in Example~\ref{exemp:RelfgGroup}    is $G$-isomorphic to the Coned-off Cayley graph $\hat\Gamma (G,\mathcal{H},X)$.
\end{example}

\begin{example}[Kr\"{o}n and M\"{o}ller's rough Cayley  graph~\cite{BM08}]
Let $G$ be a compactly generated totally disconnected locally compact group. Then $G$ contains a compact open subgroup $U$ and a finite subset $X\subset G$ such that any element of $G$ is in a left coset $wU$ where $w$  is a word in $X$, see~\cite[Lemma 2]{BM08}. The  \emph{rough Cayley  graph $\Gamma(G,U,X)$} is defined as the graph $\Gamma$ with the vertex set $V(\Gamma) = G/U$ and the edge set $E(\Gamma) = \{\{gU,gxU\} \mid g \in G, x \in X \}$; the resulting graph is a connected vertex-transitive locally finite $G$-graph~\cite[Construction 1]{BM08}. The quotient of $\Gamma$ by $G$ induces a compact generating graph of $G$ with respect to the empty collection that coincides, up to $G$-isomorphism,  with the one described in Example~\ref{exemp:TDLCfgGroup}, and $\Gamma(G,U,X)$ is the induced Cayley-Abels graph.  
\end{example}

\begin{theorem}[Characterization of Cayley-Abels graphs]\label{thm:Cayley-Abels}
Let $(G,\mathcal H)$ be a proper pair.
For a $G$-graph $\Gamma$, the following statements are equivalent:
\begin{enumerate}
    \item \label{item:CA1} Up to isomorphism of $G$-graphs, $\Gamma$ is a Cayley-Abels graph of $G$ relative to  $\mathcal{H}$ induced by a compact  generating graph in the sense of Definition~\ref{def:RelativeGraphs}.

    \item \label{item:CA2} $\Gamma$ is a connected discrete cocompact simplicial $G$-graph with compact  edge stabilizers,  vertex stabilizers are either compact or conjugates of subgroups in $\mathcal{H}$,  every $H\in \mathcal{H}$ is the $G$-stabilizer of a vertex, and any pair of vertices with the same $G$-stabilizer  $H\in\mathcal H$ are in the same $G$-orbit if $H$ is non-compact. \end{enumerate}
Moreover, if $\Gamma$ is nontrivial (it contains at  least one edge) and satisfies the above conditions, then for any vertex $v\in\Gamma$, the stabilizer $G_v$ is compact if and only if $v$ has finite degree.
\end{theorem}

 The proof of Theorem~\ref{thm:Cayley-Abels} is divided into two propositions which are the contents of subsections~\ref{subsec:CayleyGraphs} and~\ref{subsec:BassTheory} respectively. 

\subsection{From compact generating graph to discrete cocompact action on a graph.}\label{subsec:CayleyGraphs}

Versions of the following proposition can be found in the literature for discrete groups, for example see~\cite[Proof of Theorem 7.1]{ACCMP} and~\cite[Proposition 4.9]{SamLouisEdu}. 

\begin{proposition}\label{rem:RelativeCayley}
Let  $ (\mathcal{G},\Lambda,\phi)$ be a compact generating graph of $G$ relative to a finite collection $\mathcal{H}$ of open subgroups, and let $\mathcal{T}$ be the corresponding Bass-Serre tree.
 Consider the Cayley-Abels graph $\Gamma= \Gamma(\mathcal{G},\Lambda,\phi)$, the quotient map $\rho\colon \mathcal{T}\to \Gamma$, and the sequence
 \[ 1 \to \ker(\phi) \to \pi_1(\mathcal{G},\Lambda) \xrightarrow{\phi} G \to 1.     \]
 Then the following statements hold:
\begin{enumerate}
    \item The group $\ker(\phi)$ is a discrete and closed subgroup of $\pi_1(\mathcal{G},\Lambda)$ which acts freely on $\mathcal{T}$, and $\rho$ is a covering map.
    
    \item \label{item:iso} If $x$ is a vertex of $\Gamma$, then there is a group isomorphism \[\ker(\phi) \to \pi_1(\Gamma, x) \qquad g \mapsto [\gamma_g], \] where $\gamma_g=\rho\circ \alpha_g$ and $\alpha_g$ is the shortest path from $x$ to $g.x$ in $\mathcal T$.
    
    \item \label{item:G-action} The group $G$ acts on $\Gamma$ discretely, cocompactly, edge stabilizers are compact, and vertex stabilizers are either compact or conjugates of subgroups in $ \mathcal{H}$.
    
    \item For any $H \in \mathcal{H}$, there is a vertex $v \in \Gamma$ such that $G_v =  H$. 
    
    \item \label{item:Non-compact} Suppose $(G,\mathcal H)$ is a proper pair. Then any pair of vertices with the same  $G$-stabilizer   $H\in \mathcal H$ are in the same $G$-orbit if $H$ is non-compact.

    \item Suppose $G \not\in\mathcal H$. For any vertex $v$ of $\Gamma$, the $G$-stabilizer $G_v$ is compact if and only if $v$ has finite degree.
\end{enumerate}
\end{proposition}

Recall that a group action on a complex has no inversions if for every cell its setwise stabilizer coincides with  its pointwise stabilizer.

\begin{lemma} \label{lemm:infiniteValence}
Let $G$ be a group acting on a connected graph $K$ discretely, cocompactly, and with compact edge stabilizers. For any vertex $v\in K$ incident to at least one edge, $v$ has infinite degree if and only if its stabilizer $G_v$ is non-compact.
\end{lemma}
\begin{proof}
If the action has inversions, replace $K$ with its barycentric subdivision. Let $v$ be vertex in $K$ and $e$ be an edge incident to $v$, then there is bijection between the set of left cosets $G_v/G_e$  and the $G_v$-orbit of $e$ given by $gG_e \mapsto ge$. Let us suppose $v$ has infinite degree. Since the action is cocompact, there exists an edge $e$ adjacent to $v$ with an infinite orbit. Thus $[G_v:G_e]$ is infinite. Since $G_e$ is open, the cosets of $G_e$ form an infinite cover of $G_v$ with no finite subcover. Conversely suppose $G_v$ is non-compact and let $e$ be an edge adjacent to $v$. Since $G_e$ is compact,  $[G_v:G_e]$ is infinite. Thus $v$ has infinite degree.
\end{proof}

\begin{proof}[Proof of Proposition~\ref{rem:RelativeCayley}]
Let $\widetilde{G} = \pi_1(\mathcal{G},\Lambda)$ and let $\tilde{x} \in \mathcal{T}$ be any vertex. Since $\phi$ is injective on the vertex stabilizer $\widetilde{G}_{\tilde{x}}$, we have $\ker(\phi) \cap \widetilde{G}_{\tilde{x}} = 1$. Hence $\ker(\phi)$ acts freely and cellularly on $\mathcal{T}$. Therefore $\rho\colon \mathcal T \to \Gamma$ is a covering map, and $\ker(\phi)$ is a discrete and closed subgroup of $\widetilde{G}$. The isomorphism $\ker(\phi)\to \pi_1(\Gamma,x)$ is a well known consequence of covering space theory.

For item~\eqref{item:G-action},  we first observe the following properties about the $\widetilde{G}$-action on $\mathcal{T}$. The action is cocompact since $\mathcal{T}/\widetilde{G} \simeq \Lambda$ is a finite graph; and since vertex and edge stabilizers of $\mathcal{T}$ in $\widetilde{G}$ are isomorphic to vertex and edge groups, by Proposition~\ref{prop:GTopology}, they are open subgroups of $\widetilde{G}$, in particular the $\widetilde{G}$-action on $\mathcal{T}$ is discrete. Moreover, by Definition~\ref{def:GenGraph}, for the $\widetilde{G}$-action on $\mathcal T$, edge stabilizers are compact and vertex stabilizers are either compact or conjugates of subgroups in $\mathcal{H}$. Moreover, for every $H\in\mathcal{H}$, the tree $T$ contains a vertex with $\widetilde{G}$-stabilizer equal to $H$.

Let $N$ denote the normal subgroup $\ker(\phi)$ and let $\rho \colon \mathcal{T} \to \Gamma$ be the quotient map. The $\widetilde G$-action on $\mathcal{T}$ induces an action of $G =\widetilde{G}/N $ on $\Gamma = \mathcal{T}/N$.
The covering map $\rho\colon \mathcal T \to \Gamma$ is  equivariant with respect to $\phi$ and the restriction $\phi\colon \widetilde{G}_v \to G_{\phi(v)}$ is an isomorphism for any vertex $v$ of $T$.

Let $y$ be any vertex or edge in $\Gamma$, and let $\tilde{y}$ in $\mathcal{T}$ such that $\rho(\tilde{y})=y$. Then  $ \widetilde G_{\tilde{y}}N$ is an open subgroup and   $G_y = \phi(\widetilde{G}_{\tilde{y}}N) = \phi(\widetilde{G}_{\tilde{y}})$. Therefore  $G_y$ is open as $\phi$ is an open map, and if $\widetilde{G}_y$ is compact, then $G_y$ is compact as well. To summarize,  $G$-action on $\Gamma$ is discrete; is cocompact  since $\Lambda$ is finite; has compact edge stabilizers, and each vertex stabilizer is either compact or a conjugate of some $H \in \mathcal{H}$, and each $H\in\mathcal{H}$ is the $G$-stabilizer of a vertex of $\Gamma$. 

Now we  prove item~\eqref{item:Non-compact}.   Each subgroup $H\in \mathcal{H}$ is naturally identified with an isomorphic subgroup of $\widetilde G$, see Definition~\ref{def:GenGraph}; to simplify notation, we assume $H$ is a subgroup of $\widetilde{G}$ and $\phi$ restricted to $H$ is the identity map.   

Let $u_1$ and $u_2$ be vertices of $\Gamma$ with the same $G$-stabilizer $H\in \mathcal{H}$ and suppose $H$ is non-compact. Let $v_1$ and $v_2$ vertices of $\mathcal T$ mapping to $u_1$ and $u_2$ respectively. Suppose that $v_1$ has $\widetilde{G}$-stabilizer the vertex group $H_1^{g_1}$, and $v_2$ has  $\widetilde{G}$-stabilizer $ H_2^{g_2}$ where $g_1,g_2\in \widetilde{G}$ and $H_1,H_2\in \mathcal{H}$. Then $\phi(H_1^{g_1})=H$ and $\phi(H_2^{g_2})=H$; it follows that $H_1$, $H_2$ and $H$ are   subgroups in $\mathcal{H}$ that are pairwise conjugate in $G$. Since $(G,\mathcal H)$ is a proper pair, $H_1=H_2=H$. Then $g_1^{-1}v_1$ and $g_2^{-1}v_2$ have both $\widetilde{G}$-stabilizer $H$. Since $H$ is non-compact and open and all edge 
$\widetilde{G}$-stabilizer of the the tree $\mathcal{T}$ are compact open, it follows that $g_1^{-1}v_1$ and $g_2^{-1}v_2$ are the same vertex. Hence $v_1$ and $v_2$ are in the same $\widetilde{G}$-orbit, and therefore $u_1$ and $u_2$ are in the same $G$-orbit.

For the last item, suppose that $G\not\in\mathcal H$. It follows that $\Gamma$ is not a single vertex, and hence every vertex is adjacent to an edge. Then  Lemma~\ref{lemm:infiniteValence} implies that a vertex of $\Gamma$ has finite degree if and only if it has non-compact stabilizer.
\end{proof}

\subsection{From a cocompact discrete action on a graph to a compact generating graph}\label{subsec:BassTheory}

\begin{proposition}\label{prop:BassTheory}
Let $G$ be a topological group. Let   $\Gamma$ be a cocompact, connected, discrete $G$-graph with compact edge stabilizers and without inversions. 

If $\mathcal{H}$ is a set of representatives of distinct conjugacy classes of vertex stabilizers such that each non-compact stabilizer is represented. Then  there exists a compact generating graph of $G$ relative to $\mathcal{H}$ such that the induced Cayley-Abels graph (in the sense of Definition~\ref{def:RelativeGraphs}) is isomorphic to $\Gamma$ as a  $G$-graph.
\end{proposition}

\begin{lemma}\label{lem:topologyOpen}
Suppose $G, H$, and $K$ are topological groups such that the diagram  \[\begin{tikzcd}[column sep=small]
K \arrow[hookrightarrow]{r}{i}  \arrow[hookrightarrow]{rd}{j} 
& H \arrow{d}{\phi} \\
& G
\end{tikzcd}\] of homomorphisms of groups commutes. If $i,j$ are open and continuous,  then $\phi$ is open and continuous.
\end{lemma}
\begin{proof}
To prove that $\phi$ is continuous, we show that $\phi$ is continuous at 1. Let $U \subseteq G$ be an open set containing 1.  Since $j$ is continuous, $j^{-1}(U)$ is open. Let $W = i(j^{-1}(U))$. Then $W$ is open since $i$ is an open map; and by the commutative diagram,  $ 1 \in W \subseteq \phi^{-1}(U)$. Since $U$ was an arbitrary neighborhood of 1 in G, this shows that  $\phi$ is continuous at 1. 
 To prove that $\phi$ is open, first observe that if V is open in H and $V\subset i(K)$, then $i$ being continuous and $j$ being open imply that 
$\phi(V) = j( i^{-1}(V))$  is open in $G$. Moreover, for any $h\in H$, $\phi(h.V)=\phi(h).\phi(V)$ is open in $G$. For an arbitrary open subset $U$ of $H$,
let $U_h=U\cap h .i(K)$. Observe that $h^{-1}.U_h$ is open in $H$ and $h^{-1}.U_h \subset i(K)$.  Therefore $\phi (U_h)$ is open in $G$, and hence $\phi (U)=\bigcup \limits_{h\in H} \phi (U_h)$ is open in $G$.
\end{proof}

\begin{proof}[Proof of Proposition~\ref{prop:BassTheory}]
For discrete groups, the result follows from a well known construction of Bass~\cite[Section 3]{BASS1993} that we recall below. The case for topological groups follows from the same construction after addressing  topological matters. 

 Let $\Lambda$ be the quotient graph given by $\Gamma/G$, and let $r \colon \Gamma \to \Lambda$ be the quotient map. Choose a tree $T$ and a connected graph $S$ such that $T \subseteq S \subseteq \Gamma$, $r \colon T \to \Lambda$ is bijective on vertices, and $r \colon S \to \Lambda$ bijective on edges.
For any $v \in \Lambda$ and $e \in \Lambda$, the vertex group $\mathcal{G}_v$ and edge group $\mathcal{G}_e$ are defined to be the vertex stabilizer $G_{v'}$ and the edge stabilizer $G_{e'}$, where $v'$ and $e'$ are preimages of $v$ and $e$ under $r$ in $T$ and $S$ respectively. For every vertex $s \in S$, choose $g_s\in G$ such that $g_ss \in T$ and assume $g_s= 1$ if $s \in T$. Let $e$ be an edge of $\Lambda$ with endpoint vertices $u$ and $v$. Let $e'$ be the preimage of $e$ in $S$ with endpoints $u'$ and $w'$. Since $e'$ has at least one endpoint in $T$,  without loss of generality assume that $u' \in T$. Define the morphism $\mathcal{G}_e \to \mathcal{G}_u$ to be the inclusion. The morphism from $\mathcal{G}_e \to \mathcal{G}_v$ is defined as 
 $h \mapsto g_{w'}hg_{w'}^{-1} $. Thus $(\mathcal{G}, \Lambda)$ defines a graph of groups.  
Observe that if $G$ is a topological group with discrete action on $\Gamma$ then $G_{e'}$ and $G_{v'}$ are topological groups with the subspace topology and the morphisms $G_{e'} \to G_{v'}$ are open topological embeddings. In particular, $(\mathcal{G}, \Lambda)$ is a graph of topological groups.

 Fix a vertex $a \in \Lambda$ and let $a' \in \Gamma$ such that $r(a') = a$.  Let $\widetilde{G}=\pi_1(\mathcal{G},\Lambda, a)$ be the fundamental group of graph of groups, and let  $\mathcal{T}$ be the corresponding Bass-Serre tree. Then by~\cite[Theorem 3.6]{BASS1993} there is a short exact sequence

\[ 1 \to \pi_1(\Gamma,a') \to \widetilde{G} \xrightarrow{\phi} G \to 1,     \]
and a $\phi$-equivariant morphism of graphs $\rho \colon \mathcal{T} \to \Gamma $, such that for any vertex $y \in \mathcal{T}$, the restriction of $\phi$ to stabilizers $\widetilde{G}_y \xrightarrow{\phi} G_{\rho(y)}$ is an   isomorphism of groups.

Now we construct a compact generating graph of $G$ relative to $\mathcal{H}$. Consider the graph of topological groups $(\mathcal{G}, \Lambda)$ constructed above. Since $G$ acts on $\Gamma$ cocompactly, $(\mathcal{G}, \Lambda)$ is a finite topological graph of groups. The hypothesis on the  $G$-stabilizers of vertices and edges of $\Gamma$ imply that the edge groups of $(\mathcal{G}, \Lambda)$ are compact and the vertex groups are either compact or in $\mathcal{H}$.
Let $w \in \Gamma$ be any vertex and $\tilde{w}\in \mathcal{T}$ be one of its pre-images under $\rho$. Then $\widetilde{G}_{\tilde{w}} \xrightarrow{\phi} G_w$ is an isomorphism of topological groups and we have that 
\[\begin{tikzcd}[column sep=small]
\widetilde{G}_{\tilde{w}} \arrow[hookrightarrow]{r}{}  \arrow[rightarrow]{d}
& \tilde{G} \arrow{d}{\phi} \\
G_w \arrow[hookrightarrow]{r}{} & G
\end{tikzcd}\]
is a commutative diagram.

Since $\widetilde{G}$-action on $\mathcal{T}$ is discrete, $\widetilde{G}_{\tilde{w}}$ is open in $\widetilde{G}$, and thus by Lemma~\ref{lem:topologyOpen}, $\phi$ is continuous and open. Therefore $(\mathcal{G}, \Lambda , \phi)$ is a compact generating graph of $G$ relative to $\mathcal{H}$. Since $\rho \colon \mathcal{T} \to \Gamma $ is a $\phi$-equivariant morphism of graphs, the corresponding Cayley-Abels graph is isomorphic to $\Gamma$ as a $G$-graph.
\end{proof}

\subsection{Proof of Theorem~\ref{thm:Cayley-Abels}}\label{subsec:proof of TopChar}
 
That \eqref{item:CA1} implies \eqref{item:CA2} follows directly from Proposition~\ref{rem:RelativeCayley}.  
Conversely, suppose  \eqref{item:CA2}  holds. For each $H \in \mathcal{H}$, let $v_H$ be a vertex in $\Gamma$ with $G$-stabilizer $H$, and let $S=\{v_H \mid H \in \mathcal{H}\}$. Since no two distinct subgroups in $\mathcal{H}$ are conjugate, no two distinct elements of $S$ are in same $G$-orbit. For any vertex $v$ of $\Gamma$ with non-compact $G$-stabilizer, $G_v$ is conjugate to a subgroup in $\mathcal{H}$. Since any pair of vertices with the same  $G$-stabilizer $H\in \mathcal H$ are in the same $G$-orbit if $H$ is non-compact; it follows that every vertex of $\Gamma$ with non-compact stabilizer is in the $G$-orbit of an element in $S$.  By Proposition~\ref{prop:BassTheory}, there exists a compact generating graph   $(\mathcal{G}, \Lambda, \phi )$ of $G$ relative to $\mathcal{H}$ such that the corresponding Cayley-Abels graph is isomorphic to $\Gamma$. 
The statement for vertices of $\Gamma$ on the equivalence of having compact $G$-stabilizer and having finite degree in $\Gamma$ is part of the conclusion of Proposition~\ref{rem:RelativeCayley}. \qed

\section{Relative compact generation} \label{sec04:RelGen}

In this section, we state and prove a version of  Theorem~\ref{thmX:TopCharGraph}. 

\begin{definition}[Relative Compact Generation] A topological group $G$ is \emph{compactly generated relative to a collection of subgroups} $\mathcal H$ if there is a compact subset $A\subset G$ such that $G$ is algebraically generated by $ A\cup\bigcup\mathcal{H}$; in this case we say that $A$ is a \emph{compact generating set of $G$ relative to $\mathcal H$}. 
\end{definition}

The \v{S}varc-Milnor Lemma implies that a discrete group is finitely generated if and only if it acts cellularly, cocompactly and with finite vertex stabilizers on a connected graph~\cite[Proposition 8.19]{BrHa99}. The main result of this section, Theorem~\ref{prop:TopCharGraph}, generalizes this fact for proper pairs $(G, \mathcal H)$, see Definition~\ref{def:properPair}.

A $G$-action on a cell complex $X$ is called \emph{discrete} if it is a cellular action such that the pointwise stabilizer of each cell is an open subgroup of $G$. A graph is a 1-dimensional cell complex.  

\begin{theorem}[Topological Characterization] \label{prop:TopCharGraph} 
Let $(G,\mathcal H)$ be a proper pair. The following statements are equivalent:
\begin{enumerate}
    \item \label{TopoChar01} There is a compact generating set of $G$ relative to $\mathcal{H}$.
    \item \label{TopoChar02} There exists a compact generating graph of $G$ relative to $\mathcal{H}$.
    \item \label{TopoChar03} There is a connected discrete cocompact simplicial $G$-graph $\Gamma$ with compact  edge stabilizers,  vertex stabilizers are either compact or conjugates of subgroups in $\mathcal{H}$,  every $H\in \mathcal{H}$ is the $G$-stabilizer of a vertex, and any pair of vertices with the same $G$-stabilizer $H\in\mathcal H$ are in the same $G$-orbit if $H$ is non-compact.
\end{enumerate}
 \end{theorem}

Let us record two immediate consequences of Theorem~\ref{prop:TopCharGraph}.

\begin{corollary}[Kr\"{o}n and M\"{o}ller]\cite[Corollary 1]{BM08}\label{thm:C.A.-Graph}
     Let $G$ be a totally disconnected locally compact group. Suppose G acts on a connected locally finite graph $\Gamma$ such that the stabilizers of vertices are compact open subgroups and $G$ has only finitely many orbits on $V(\Gamma)$. Then $G$ is compactly generated.
\end{corollary}

 Remark~\ref{rem:GenGraph} implies that  a topological group  admitting a compact generating graph with respect to the empty collection contains a compact open subgroup, therefore:    

\begin{corollary}\label{cor:CompactGenerationEmptyCollection}
Let $G$ be a topological group. The following statements are equivalent:
\begin{enumerate}
    \item $G$ contains a compact open subgroup and a compact generating set.
    \item $G$ admits a compact generating graph relative to the empty collection.
\end{enumerate}
\end{corollary}

\subsection{Proof of Theorem~\ref{prop:TopCharGraph}}\label{subsec:TopCompactGen} The equivalence of \eqref{TopoChar02} and \eqref{TopoChar03} is a direct consequence of Theorem~\ref{thm:Cayley-Abels}.
The equivalence of \eqref{TopoChar01} and \eqref{TopoChar03} follows from Proposition~\ref{prop:LastOne} stated below.

\begin{proposition}\cite[Theorem 8.10]{BrHa99}\label{thm:Bridson}
Let $X$ be a topological space, let $G$ be a group acting on $X$ by homeomorphisms, and let $U$ be an open subset such that $X= GU$. If $X$ is connected, then the set $S = \{g \in G \mid g.U \cap U \neq \emptyset \}$ generates $G$.
\end{proposition}

\begin{proposition}\label{prop:LastOne}
Let $(G, \mathcal{H})$ be a proper pair. The following statements are equivalent.
\begin{enumerate}
    \item  There is a compact subset $A\subset G$ such that $G=\langle A\cup \bigcup \mathcal{H}\rangle$.
    \item There is a Cayley-Abels graph $\Gamma$ of $G$ relative to $\mathcal{H}$.
\end{enumerate}
\end{proposition}
\begin{proof}
Let $U$ be a compact open subgroup of $G$. Suppose there is a relative compact generating set $A$ of $G$ with respect to $\mathcal{H}$.  Compactness of $A$ implies that  there is a finite subset $S\subset G$ such that $A\subset SU$.  Define $\Gamma$ as the $G$-graph with vertex set $V(\Gamma)=G/U \cup G/\mathcal{H}$ and edge set $E(\Gamma)=\{ \{gU, gsU\} \mid g\in G,\ s\in S \} \cup \{\{gU, gH\}\mid g\in G,\ H\in\mathcal{H}\}$. Note that $G$ acts discretely, there are only finitely many orbits of vertices and edges on $\Gamma$, edge stabilizers are compact open, vertex stabilizers are conjugates of $U$ or $H\in \mathcal{H}$; moreover, for every $H\in\mathcal{H}$ there is a vertex in $\Gamma$ with stabilizer equal to $H$.

It is left to prove that $\Gamma$ is connected. Observe that it is enough to show that there is a  path from the vertex $U$ to the vertex $gU$ for every $g\in G$. We argue by induction on the length of $g$ as a word in the generating set $S\cup \bigcup \mathcal{H}$. Observe that if there is a path $\alpha$ in $\Gamma$ from $U$ to $gU$ and $x\in S\cup \bigcup \mathcal{H}$ then there is a path from $U$ to $gxU$, namely the concatenation of the paths $\alpha$ and $g\beta$ where $\beta$ is a path from $U$ to $xU$ given by  
\[\beta=[U, H, xU],\qquad \text{ or }\qquad  \beta=[U,xU]\]
if $x\in H \in \mathcal{H}$ or $x\in S$ respectively. By Theorem~\ref{thm:Cayley-Abels}, $\Gamma$ is a   Cayley-Abels graph of $G$ relative to $\mathcal{H}$.

Conversely, suppose  $\Gamma$ is a  Cayley-Abels graph  of $G$ relative to $H$. Since $\Gamma$ is connected and cocompact, Proposition~\ref{thm:Bridson} implies that $G$ is generated by  finite number of vertex stabilizers $G_{v_1},\ldots , G_{v_k}$ and a finite set $S$. Then $S$ together with the union of the $G_{v_i}$ that are compact  is a compact generating set of $G$  relative to $\mathcal{H}$.
\end{proof}

\section{Equivariant edge attachments and fineness} \label{sec:appendix}

This section revisits an argument from~\cite{MR21} in order to prove that certain natural extensions of $G$-graphs preserve fineness. The main result of this section is Theorem~\ref{thm.2hs}. Throughout the section, $G$ denotes a topological group. A metric space is \emph{locally finite} if balls of finite radius are finite. All graphs in this section are 1-dimensional simplicial complexes. Let $\Gamma$ be a simplicial graph, let $v$ be a vertex of  $\Gamma$, and let \begin{align*}
T_v \Gamma = \{w \in V(\Gamma) \mid \{v,w\}\in E(\Gamma)\}.
\end{align*}
denote the set of the  vertices adjacent to $v$.
For $x,y \in T_v \Gamma$,  the \emph{angle metric} $\angle_{T_v\Gamma}  (x,y)$ is the combinatorial length of the shortest path in the graph  $\Gamma - \{v\}$ between $x$ and $y$, with $\angle_{T_v\Gamma} (x,y) = \infty$ if there is no such path. 

\begin{definition}[Bowditch fineness]\cite{Bo12} A simplicial graph $\Gamma$ is \emph{fine at $v$} if $(T_v \Gamma,\angle_{T_v\Gamma})$ is a locally finite metric space. 
A graph $\Gamma$ is   \emph{fine} if it is fine at every vertex. 
\end{definition}

For the equivalence between this definition of fineness and the one given in the introduction of this article see~\cite[Prop. 2.1]{Bo12}.

\begin{definition}[Equivariant attachment of edges]\label{def:uH-Attachment}
Let $G$ be a group and let $\Gamma$ and $\Delta$ be simplicial  $G$-graphs. 
\begin{enumerate}
    \item  Let $u \in V(\Gamma)$ and let $H \leq G$ be a subgroup.  The  $G$-graph $\Delta$ is \emph{obtained from $\Gamma$ by attaching an edge  $G$-orbit with representative $\{u,H\}$} if \[V(\Delta) = V(\Gamma) \sqcup G/H, \qquad 
 E(\Delta) = E(\Gamma) \sqcup \{\{gu,gH\}| g \in G\}\]
 where $G/H$ denotes the $G$-set of left cosets of $H$ in $G$.
 
 \item Let $u,v\in V(\Gamma)$ distinct vertices. The  $G$-graph $\Delta$ is \emph{obtained from $\Gamma$ by attaching a $G$-orbit of edges with representative $\{u,v\}$} if 
 \[ V(\Delta) = V(\Gamma),\qquad E(\Delta)  = E(\Gamma) \cup \left\{\{g.u,g.v\} \mid g \in G \right\}  \]
 \end{enumerate}
\end{definition}

 \begin{theorem}  \label{thm.2hs}
Let  $\Gamma$ be a connected discrete simplicial $G$-graph  with compact edge stabilizers. Let   $u, v, a \in V(\Gamma)$, and  let $H\leq G$ be a compact open subgroup. Let $\Delta$ be a $G$-graph obtained from $\Gamma$
\begin{enumerate}
     \item by attaching   a $G$-orbit of edges with representative $\{u,v\}$; or
    \item by attaching a $G$-orbit of edges with representative $\{u,H\}$.
\end{enumerate}
Then  $\Gamma$ is fine at   $a$ if and only if $\Delta$ is fine at $a$. Moreover, the inclusion $\Gamma \hookrightarrow \Delta$ is a quasi-isometry.
\end{theorem} 

Since fineness at a vertex is preserved under taking subgraphs, the if part of the theorem is trivial. Before moving to the proof of the theorem, let us state and prove a corollary.

\begin{corollary}\label{prop:EdgeAttachments}
 Let $\Gamma_1$ and $\Gamma_2$ be  cocompact connected discrete  simplicial $G$-graphs with compact edge stabilizers.
 Let $\mathcal{V}_\infty(\Gamma_i)$ be the  set of vertices of $\Gamma_i$ with non-compact stabilizer. If $\mathcal{V}_\infty(\Gamma_i)$ is empty for $i=1,2$, or  
 there is a $G$-equivariant bijection $\mathcal{V}_\infty(\Gamma_1) \xrightarrow{\eta} \mathcal{V}_\infty(\Gamma_2)$, then:
 \begin{enumerate}
    \item $\Gamma_1$ and $\Gamma_2$ are quasi-isometric, and
    \item $\Gamma_1$ is fine if and only if $\Gamma_2$ is fine.
\end{enumerate}
\end{corollary}
\begin{proof}
Suppose that there is a $G$-equivariant bijection $\mathcal{V}_\infty(\Gamma_1) \xrightarrow{\eta} \mathcal{V}_\infty(\Gamma_2)$. 
Let $\Gamma$ be the simplicial $G$-graph obtained by taking disjoint union of $\Gamma_1$ and $\Gamma_2$ identified along $\mathcal{V}_\infty(\Gamma_1)$ and $ \mathcal{V}_\infty(\Gamma_2)$ through $\eta$. Since $\mathcal{V}_\infty(\Gamma_i)$ is non-empty, $\Gamma$ is a connected $G$-graph. By cocompactness of $\Gamma_2$, the $G$-graph $\Gamma$  can be constructed from $\Gamma_1$ by finite sequence of  $G$-edge attachments. Since $\mathcal{V}_\infty(\Gamma_2)$ contains all vertices of $\Gamma_2$ with non-compact stabilizer, we only need to perform equivariant edge attachment of edges satisfying the hypothesis of Theorem~\ref{thm.2hs}. By induction, $\Gamma$ is fine if and only if $\Gamma_1$ is fine; and the inclusion $\Gamma_1\hookrightarrow \Gamma$ is a quasi-isometry. By symmetry,  $\Gamma_2$ is fine if and only if $\Gamma$ is fine, and  $\Gamma_2$ is quasi-isometric to $\Gamma$.

Suppose that $\mathcal{V}_\infty(\Gamma_i)$ is empty for $i=1,2$.  By Lemma~\ref{lemm:infiniteValence}, $\Gamma_i$ is a locally finite graph and hence is fine. To prove that $\Gamma_1$ and $\Gamma_2$ are quasi-isometric we  use an argument similar to the previous paragraph. Let $u_i$ be a vertex of $\Gamma_i$ and let $\Gamma$ be the $G$-graph obtained as the disjoint union of $\Gamma_1$, $\Gamma_2$ together with the $G$-set of edges $\{ \{gu_1, gu_2\} \mid g\in G\}$. Observe that $\Gamma$ is connected cocompact $G$-graph such that every vertex has compact $G$-stabilizer. It follows that $\Gamma$ can be constructed from $\Gamma_1$ by a finite sequence of $G$-edge attachments satisfying the hypothesis of Theorem~\ref{thm.2hs} and therefore $\Gamma_1\hookrightarrow \Gamma$ is a quasi-isometry.
\end{proof}

There is a version of the previous corollary  for pairs $(G,\mathcal H)$ with $G$ discrete in~\cite[Proposition 5.6]{SamLouisEdu}.

Theorem~\ref{thm.2hs} addresses two constructions that preserve fineness of vertices, we refer to them as the first and second construction according to the enumeration in the statement. There are versions of Theorem~\ref{thm.2hs}  in the case of the {\bf first construction} and under the assumption that $G$ is discrete:
\begin{enumerate}
    \item Bowditch shows that if $\Gamma$ has finitely many $G$-orbits of vertices and edges and is fine, then $\Delta$ is fine; see~\cite[Lemma 4.5]{Bo12}.  An alternative argument (that works only in the torsion-free case) for this statement can be found in ~\cite{MaWi10}. 
    \item The statement of Theorem~\ref{thm.2hs}  for the first construction in the case that $G$ is discrete can be found in~\cite[Proposition 4.2]{MR21}. 
\end{enumerate}
 The proof of    Theorem~\ref{thm.2hs} for both constructions follows the same strategy as the argument in \cite[Proof of Proposition 4.2]{MR21}.   We only prove  Theorem~\ref{thm.2hs} for the second construction, i.e., the case that $\Delta$ is obtained by attaching a $G$-orbit of edges with representative $\{u,H\}$. This case has not been   addressed even in the case that the group is discrete. While there is a significant overlap with the argument in~\cite[Proof of Proposition 4.2]{MR21}, we decided to include a complete proof since there is number of additional lemmas that are required besides addressing   topological matters arising from replacing  finiteness of edge stabilizers   with the assumption that edge stabilizers are compact and open.
 For the convenience of the reader we included some arguments from~\cite{MR21} in some cases almost verbatim.

\subsection{Preliminaries}

Let us fix some notation for paths in a simplicial graph $\Gamma$. A \emph{path} or an \emph{edge-path} from a vertex $v_0$ to a vertex $v_n$ of   $\Gamma$ is a sequence of vertices $[v_0, v_1 \dots , v_n]$, where ${v_i}$ and $v_{i+1}$ are adjacent (in particular distinct) vertices for all $i \in \{0, \dots , n-1 \}$. The path is \emph{embedded} if all vertices of the path are distinct. The \emph{length} of a path is  the total number of vertices in the sequence minus one. A path of length $k$ is called a \emph{$k$-path}. If $\alpha=[u_1,\ldots ,u_k]$ and $\beta=[v_1,\ldots v_\ell]$ are paths with $u_k=v_1$, then $[\alpha, \beta]$ denotes the concatenated path $[u_1,\ldots,u_k,v_2,\ldots v_\ell]$.

In this section we  use an equivalent formulation of fineness from~\cite{MR21} that we described below. 

A  path $[u, u_1 \dots , u_k]$ in a graph  $\Gamma$ is an  \emph{escaping path from $u$ to $v$} if $v=u_k$  and $u_i \neq  u$ for every $i \in \{1, \dots , k\}$.
For vertices $u$ and $v$ of $\Gamma$ and $k \in \mathbb{Z}_{+}$,  define:
\[
\begin{split}
\vec{uv} (k)_\Gamma = \{w \in T_u \Gamma \mid w \text{ belongs } & \text{to an escaping } \\
& \text{path from }u\text{ to }v\text{ of length $\leq k$}\}.
\end{split}
\]

\begin{remark}\label{rem:thanks}
If $w\in \vec{uv}(k)_\Gamma$ then there is an escaping path $\delta$  from $u$ to $v$ of length at most $k$ such that $w$ is the vertex immediately after $u$. Indeed, if $\gamma$ is an escaping path $[u,u_1,\cdots, u_i, \cdots, u_m]$ from $u$ to $v$, $m\leq k$ and $w=u_i$, then $[u,u_i,\ldots , u_m]$ is such an escaping path.  
\end{remark}

\begin{proposition}\cite[Lemma 4.4]{MR21}\label{lem:4.4Farhan}
A graph $\Gamma$ is fine at $u \in V(\Gamma)$ if and only if $\vec{uv}(k)_\Gamma$ is a finite set for every integer $k>0$ and every vertex $v \in V(\Gamma)$. 
\end{proposition}

\begin{remark}\label{rem:simplification}
For vertices $u$ and $v$ of $\Gamma$ and $k \in \mathbb{Z}_{+}$, 
\[ \vec{uv}(k+1)_\Gamma =  \bigcup \left\{ \vec{uw }(k)_\Gamma \colon w\in T_v\Gamma \right\}. \]
\end{remark}

We conclude the preliminaries of this section with the following lemma.

\begin{lemma}\label{lem:compactness}
Let $\Gamma$ be a connected discrete $G$-graph with compact edge stabilizers. Let $c=[x,y,z]$ be a 2-path and let $e=[u,v]$ be a 1-path.   
The set
\[  B=\{  g.z\in V(\Gamma)  \mid g\in G,\  g.x=u,\   g.y=v  \} \]
is finite.
\end{lemma}
\begin{proof}
Suppose $B$ is non empty, and let 
$A=\{g\in G \colon g.x=u,\   g.y=v \}$.
 Fix an element $g_0\in A$ and observe that $A = g_0 (G_x\cap G_y)$.  Since $\Gamma$ has compact edge stabilizers, $G_x\cap G_y$ is a compact subgroup, and therefore $A$ is compact. 
 
 Since $\Gamma$ is a discrete $G$-graph, it follows that $G_x\cap G_y\cap G_z$ is an open subgroup.   Observe that   
\[A= \bigcup_{g\in A} g(G_x\cap G_y\cap G_z),\]
and hence, by compactness, there is a finite set $\{g_1,\ldots , g_n\}$ such that $A=\bigcup_{i=1}^n g_i(G_x\cap G_y\cap G_z)$. It follows that $B=\{g_1.z, g_2.z, \ldots , g_n.z\}$
\end{proof}

\subsection{Proof of Theorem~\ref{thm.2hs}, fineness part}

Let  $\Gamma$ be a connected discrete $G$-graph  with compact edge stabilizers. Let $u \in V(\Gamma)$, let $H\leq G$ be a compact open subgroup, and let $\Delta$ be the $G$-graph obtained from $\Gamma$ by attaching a new $G$-orbit of an edge with representative $\{u, H\}$.

\begin{lemma}
$\Delta$ is connected discrete $G$-graph with compact edge stabilizers.
\end{lemma}
\begin{proof}
It is an observation that $\Delta$ is connected and the vertex $H$ of $\Delta$ has $G$-stabilizer the subgroup $H$ which is open by assumption.  The edge $\{u,H\}$ of $\Delta$ has $G$-stabilizer $G_u\cap H$ which is open since both $G_u$ and $H$ are open, and it is compact since $H$ is compact by assumption. Since any vertex or edge of $\Delta$ which is not in the $G$-orbits of the vertex $H$ or the edge $\{u,H\}$ is in $\Gamma$, we have that $\Delta$ is a discrete $G$-graph.
\end{proof}

\begin{lemma}[The vertex $H$ has finite degree]\label{lem:001}
The set $T_H\Delta$ of   vertices of $\Delta$ adjacent to the vertex $H$ is a finite subset of $V(\Gamma)$.
\end{lemma}
\begin{proof}
By definition of $\Delta$, every vertex  adjacent to the vertex $H$ is a vertex of $\Gamma$. Observe that the vertex $H$ of $\Delta$ has stabilizer the subgroup $H$, and the edge $\{u,H\}$ has stabilizer $G_u\cap H$.
Since all vertices of $\Delta$ adjacent to the vertex $H$ are in the $G$-orbit of $u$, it follows that the vertex $H$  has degree equal to  the index of the  subgroup $H\cap G_u$ in the group $H$. Since $H$ is compact and $H\cap G_u$ is an open subgroup, the vertex $H$ of $\Delta$ has finite degree. \end{proof}
 
\begin{lemma}[Fineness criterion for $\Delta$]\label{lem:002}
If $\vec{ab}(k)_{\Delta}$ is  finite for  every $k\geq 1$ and every  $b \in V(\Gamma)$, then $\Delta$ is fine at $a$.
\end{lemma}
\begin{proof}
Let $v$ be a vertex of $\Delta$ and let $k\geq 1$. If $v$ is not in the $G$-orbit of the vertex $H$, then by assumption $\vec{av}(k)_{\Delta}$ is  finite for  every $k\geq 1$.
Suppose that $v$ is in the $G$-orbit of the vertex $H$. Then Lemma~\ref{lem:001} implies that $v$ has finite degree and $T_v\Delta=\{b_1,\ldots , b_m\}$ is a subset of $V(\Gamma)$. Then Remark~\ref{rem:simplification} implies that  
\[  \vec{av}(k)_{\Delta} = \bigcup_{i=1}^m \vec{ab_i}(k-1).    \]
Since each $b_i\in V(\Gamma)$, the hypothesis implies that $\vec{ab_i}(k-1)$ is a finite set for each $b_i$. Therefore $\vec{av}(k)_{\Delta}$ is a finite set for every $v\in V(\Delta)$ and $k\geq1$. By Proposition~\ref{lem:4.4Farhan}, $\Delta$ is fine at $a$.
\end{proof}

Suppose that  $\Gamma$ is fine at the vertex $a$. 
Let $b$ be a vertex of $\Gamma$ and let $k\geq 1$. We prove below that $\vec{ab}(k)_\Delta$ is finite. Observe that this implies that $\Delta$ is fine at $a$ in view of Lemma~\ref{lem:002}. The argument follows the skeleton of the proof of~\cite[Proposition 4.2]{MR21}. The rest of this subsection proves that $\vec{ab}(k)_\Delta$ is finite.
 
\subsubsection*{The paths $\alpha_{ij}$ and the constants $\ell$ and $n$.}

By Lemma~\ref{lem:001}, 
\[T_H\Delta=\{v_1,v_2,\ldots , v_m\} \subset V(\Gamma).\] 
For each $v_i,v_j\in T_H\Delta$, let $\alpha_{ij}$ be a minimal length (embedded) path from $v_i$ to $v_j$ in $\Gamma$, note that such a path exists since $\Gamma$ is connected. Let $\ell$ be an upper bound for the length of the paths $\alpha_{ij}$, that is
\[  |\alpha_{ij}| \leq \ell \]
for any $v_i,v_j\in T_H\Delta$. Let 
\begin{equation}\label{eq:defn}  n = k\ell \end{equation}

\subsubsection*{The finite sets $W_i$ and $Z_i$.}
A subpath of length two of a path $P$ is called a  \emph{corner of $P$}.  Let 
\begin{align*}
W_n &= \vec{ab}(n)_\Gamma.
\end{align*}
Let $j\leq n$ and suppose $W_j$ has been defined. Let 
\begin{equation}\nonumber
    \begin{split}
    Z_{j-1} = W_j \cup  \{z \in T_a \Gamma  \mid  \exists & w\in W_j\ \exists g\in G \  \exists v_i,v_j \in T_H\Delta\     \\
    & \exists c \text{ corner of }  \alpha_{ij}\  \text{such that } g.c = [z,a,w] \}.
    \end{split}
\end{equation}
\begin{align*}
W_{j-1} &= \{w \in T_a \Gamma \mid \exists z \in Z_{j-1} \text{ such that } \angle_{T_a \Gamma} (z,w) \leq n \}.
\end{align*}
Observe that  
\begin{align} \label{w}
W_j \subseteq Z_{j-1} \subseteq W_{j-1}, \qquad \text{for all } 1\leq j \leq n
\end{align}

\begin{lemma} \label{lem:WtoZ}
If $W_j$ is finite, then $Z_{j-1}$ is finite.
\end{lemma}
\begin{proof} This is a consequence of the assumption that $G$ acts discretely on $\Gamma$ and edges have compact $G$-stabilizers.  By contradiction, assume that $Z_{j-1}$ is infinite and  $W_j$ is finite. Since there are finitely choices for $\alpha_{ij}$ and each of these paths has finitely many corners,  the pigeon-hole argument shows that there is $w\in W_j$, there is an $\alpha_{ij}$, and there is  corner $c=[v,x,y]$ of $\alpha_{ij}$ such that the set
\[ B=\{ g.v\in V(\Gamma) \mid g\in G,\quad g.x=a,\quad g.y=w  \}\]
is infinite. Since $\Gamma$ is a discrete $G$-graph with compact edge stabilizers, Proposition~\ref{lem:compactness} implies that $B$ is finite, a contradiction.
\end{proof}

\begin{lemma}\label{lem:ZtoW}\label{g1}
For $1\leq j< n$, $W_j$ and $Z_j$ are finite subsets of $T_a\Gamma$. In particular, $W_1$ is finite.
\end{lemma}
\begin{proof}\cite[Proof of Lemma 4.6]{MR21}. The conclusion follows by an inductive argument using Lemma~\ref{lem:WtoZ} and the following pair of claims.

\emph{Claim 1: If $Z_j$ is finite, then $W_j$ is finite.} Since $\Gamma$ is fine at $a$,  for each $z\in Z_{j-1}$, there are finitely $w\in T_a\Gamma$ such that $\angle_{T_a\Gamma}(w,z)\leq n$. Hence if  $Z_j$ is finite, then $W_j$ is finite.

\emph{Claim 2. $W_n$ is finite.}  By hypothesis,  $\Gamma$ is fine at $a$. Then   Lemma~\ref{lem:4.4Farhan} implies that $\vec{ab}(n)_\Gamma = W_n$ is  finite.
\end{proof}

\subsubsection*{Projecting  paths from $\Delta$ to $\Gamma$.}
 Let $\delta$ be a  $k$-path in $\Delta$ with initial vertex  in $\Gamma$. 
An \emph{$\alpha$-replacement of $\delta$} is a path
$\gamma$ in $\Gamma$ obtained as follows:
Replace each corner of $\delta$ of the form $[g.v_i, gH, g.v_j]$ for some $v_i,v_j\in T_H\Delta$ and $g \in G$ by the path $g.\alpha_{ij}$ (make a choice of $g$ if necessary); if the terminal vertex of the resulting path is not in $\Gamma$ then remove that vertex. Observe that $\gamma$ is a path of length at most $n$; recall $n$ is  defined in~\eqref{eq:defn}. 

\begin{lemma}  \label{g}
Let $\delta$ be an escaping $k$-path in $\Delta$ from $a$ to $b$ and let $\gamma$ be an $\alpha$-replacement.  Then
\[ \delta\cap T_a\Gamma    \subseteq \gamma \cap T_a\Gamma \subseteq W_1,\]
where $\delta \cap T_a\Gamma$ is the set of vertices of $\delta$ that belong to $T_a
\Gamma$, and $\gamma\cap T_a\Gamma$ is defined analogously.
\end{lemma}
\begin{proof} 
The following argument is taken almost verbatim  from~\cite[Proof of Lemma 4.7]{MR21}. It only requires minor modifications due to our definition of $\alpha$-replacement.

By construction,  $\delta\cap T_a\Gamma\subseteq \gamma \cap T_a\Gamma$.  Observe that $\gamma$ is a path  of the form \[ \gamma=[a,\gamma_1, a,\gamma_2, a, \dots ,a, \gamma_m],\] where each $\gamma_i$ is a path that does not contain the vertex $a$. Note that $m\leq n$ and that $\gamma$ is not escaping when $m>1$.  In order to prove $\gamma\cap T_a\Gamma \subseteq W_1$ is enough to show that  $\gamma_i \cap T_a\Gamma \subset W_i$ for $1\leq i\leq m$ in view of~\eqref{w} and  that  $\gamma \cap T_a\Gamma =\bigcup_{i=1}^m\gamma_i \cap T_a\Gamma$.

Let $w_i$ and $z_i$ denote the initial and terminal vertices of $\gamma_i$, respectively. The main observation: since $\delta$ is an escaping path from $a$, it follows that the corner $[z_i,a,w_{i+1}]$ of $\gamma$ is a translation of a corner of a path  $\alpha_{ij}$ for some $v_i,v_j\in T_H\Delta$.

\emph{Claim 1. $w_m \in W_m$.} Note that  $[a,\gamma_m]$ is an escaping path of length at most $n$ from $a$ to $b$ in $\Gamma$. Therefore $w_m \in \vec{ab}_\Gamma (n) = W_n$. Since $m\leq n$, it follows that  $w_m \in W_n \subseteq W_m$.

\emph{Claim 2. $z_{m-1} \in Z_{m-1}$.} Note that 
$[z_{m-1},a,w_m]$ is the translation of a corner of $\alpha_{ij}$;   since $w_m \in W_m$, we have that $z_{m-1} \in Z_{m-1}$.

\emph{Claim 3. If $z_{i} \in Z_{i}$  then $w_{i}\in W_{i}$}.
Indeed, since 
\[\angle_{T_a \Gamma} (z_{i}, w_{i}) \leq |\gamma_i |\leq  n \] and $z_{i} \in Z_{i}$, it follows that $w_{i} \in W_{i}$.

\emph{Claim 4. If $w_{i+1}\in W_{i+1}$ then $z_i\in Z_i$.}  As $[z_{i},a,w_{i+1}]$ is the translation of a corner of an $\alpha_{ij}$, if  $w_{i+1} \in W_{i+1}$, then by  definition we have that $z_{i} \in Z_{i}$.

\emph{Claim 5. If $z_{i} \in Z_{i}$  then $\gamma_{i} \cap T_a\Gamma$ is a subset of $W_{i}$.}
Let $x\in \gamma_{i} \cap T_a\Gamma$. Observe that $\angle_{T_a \Gamma} (z_{i}, x) \leq n$. Since $z_{i} \in Z_{i}$, it follows that $x \in W_{i}$.

To conclude, observe that the first four claims imply that $z_i\in Z_i$ for $1\leq i\leq m$. Then the last claim implies that $\gamma_i\cap T_a\Gamma$ is a subset of $W_i \subset W_1$ for $1\leq i\leq m$.  
\end{proof}

\subsubsection*{The finite set $X_0$}
Let 
\[X_0 = \{ x\in \vec{ab}(k)_{\Delta} \mid x \not\in T_a\Gamma \}.\]

\begin{lemma}\label{lem:eduardo}
$X_0$ is a finite set. 
\end{lemma}
\begin{proof}
If $x\in X_0$ then $x$ is a translate of the vertex $H$ of $\Delta$. Hence if $X_0$ is nonempty, then   $a$ is adjacent to a translate of the vertex $H$. Suppose that $X_0$ is nonempty and without loss of generality, assume that $a$ is adjacent to $H$, specifically
\[ T_H\Delta=\{v_1,v_2\ldots ,v_m\},\qquad \text{and}\qquad a\in T_H\Delta .\]

Suppose that $X_0$ is infinite. For each $x\in X_0$ choose an escaping path $\delta_x$ from $a$ to $b$ in $\Delta$ of length at most $k$ that contains $x$ as the vertex immediately after $a$, see Remark~\ref{rem:thanks}. Then there is $g_x\in G$ and $v_{i_x}, v_{j_x} \in T_H\Delta$ such that  $[g_x.v_{i_x}, g_xH, g_x.v_{j_x}]$ is the initial 2-subpath of $\delta_{x}$. 

Since $T_H\Delta$ is finite, by the Pigeon hole argument there is a fixed pair, after renumbering if necessary, $v_1,v_2 \in T_H\Delta$ such that 
\[  X_1=\{x \in X_0 \mid v_{i_x}=v_1 \text{ and } v_{j_x}=v_2 \}  \]
is an infinite set. Observe that 
\[  g_x.v_1 = a \text{ and }  g_x.H=x    \qquad \text{for all $x\in X_1$}. \]

For $x\in X_1$, let $\gamma_x$ be an $\alpha$-replacement of $\delta_x$ that has $g_x.\alpha_{12}$ as an initial subpath. By Lemma~\ref{g},
\[ \gamma_x \cap T_a\Gamma \subset W_1 .\]
We showed that $W_1$ is finite, see Lemma~\ref{g1}.  By the Pigeon-hole argument, there is a $z\in T_a\Gamma$ such that 
\[ X_2=\{ x\in X_1 \colon \text{$\gamma_x$ has $z$ as its second vertex.}\} \]
is an infinite set. Let $w$ be the second vertex of $\alpha_{12}$. Then  for every $x\in X_2$, $g_x.v_1=a$ and $g_x.w=z$. Since $g_x.H=x$ for all $x\in X_2$, it follows that 
\[ C=\{ gH \in V(\Delta) \mid g\in G,\quad g.w=z,\quad   g.v_1=a\}\]
is an infinite set. Since $\Delta$ is a discrete $G$-graph with compact edge stabilizers, $[w,v_1, H]$ is a 2-path in $\Delta$, and $[a,z]$ is a 1-path of $\Delta$;
Lemma~\ref{lem:compactness} implies that  $C$ is finite, a contradiction.
\end{proof}

\subsubsection*{Conclusion.}
It is left to verify that $\vec{ab}(k)_{\Delta}$ is finite. Let $\delta$ be an escaping path from $a$ to $b$ in $\Delta$ of length $\leq k$. Then Lemma~\ref{g} implies that $\delta\cap T_a\Gamma \subseteq W_1$. Since any vertex in $\delta\cap T_a\Delta$ is either in $T_a\Gamma$ or $X_0$, it follows that $\delta\cap T_a\Delta \subseteq W_1\cup X_0$. Since $\delta$ was arbitrary, \[ \vec{ab}(k)_\Delta \subseteq W_1\cup X_0.\]
By Lemmas~\ref{g1} and~\ref{lem:eduardo}, $W_1\cup X_0$ is a finite set, and hence  $\vec{ab}(k)_{\Delta}$ is finite.

\subsection{Proof of Theorem~\ref{thm.2hs}, Quasi-isometry}\label{rem:quasi-isometry}
 
Under the assumptions of Theorem~\ref{thm.2hs}, the  inclusion $\Gamma\hookrightarrow\Delta$ is a $G$-equivariant quasi-isometry.  Indeed, for any path $\delta$ in $\Delta$ with endpoints in $\Gamma$, its $\alpha$-replacement $\gamma$ is a path in $\Gamma$ with the same endpoints and $|\gamma|\leq \ell|\delta|$. Hence for any pair of vertices $u,v$ of $\Gamma$, $\dist_\Gamma(u,v)\leq \ell \dist_\Delta(u,v)\leq \ell\dist_\Gamma(u,v)$. On the other hand, every vertex of $\Delta$ is adjacent to a vertex of $\Gamma$.

\section{Invariance of Cayley-Abels graphs} \label{sec:QI-Fineness}

The main results of this section are Theorem~\ref{thm:uniqueCAgraph} and Corollary~\ref{cor:QI-Fineness}. These results generalize the fact that any two Cayley graphs with respect to finite generating sets of a group are quasi-isometric~\cite[Chapter I, Example 8.17(3)]{Br99}.

\begin{theorem}\label{thm:uniqueCAgraph}
Let $(G,\mathcal{H}_1)$ and $(G,\mathcal{H}_2)$ be proper pairs. Suppose the symmetric difference of $\mathcal{H}_1$ and $\mathcal{H}_2$ consists only of compact subgroups.

If $\Gamma_1$ and $\Gamma_2$ are  Cayley-Abels graphs of $G$ with respect to $\mathcal{H}_1$ and $\mathcal{H}_2$ respectively,  then 
\begin{enumerate}
    \item $\Gamma_1$ and $\Gamma_2$ are quasi-isometric; and
    \item $\Gamma_1$ is fine if and only if $\Gamma_2$ is fine.
\end{enumerate}
\end{theorem}

The following corollary follows directly from Theorem~\ref{thm:uniqueCAgraph}. 

\begin{corollary}[Quasi-isometry Invariance of Cayley-Abels Graphs]\label{cor:QI-Fineness}
Let $(G,\mathcal H)$ be a proper pair. 
\begin{enumerate}
    \item Any two  Cayley-Abels graphs of $G$ with respect to $\mathcal{H}$ are quasi-isometric.
    \item If one Cayley-Abels graph  of $G$ with respect to $\mathcal{H}$   is fine, then all are fine.
\end{enumerate}
\end{corollary}

\subsection{Proof of Theorem~\ref{thm:uniqueCAgraph}}
Let us fix some notation. For a graph $\Gamma$, let  $\mathcal{V}_\infty(\Gamma)$ denote its set of vertices with infinite degree. For a collection of subgroups $\mathcal{H}$ of $G$, let $\mathcal{H}_{\infty}= \{H \in \mathcal{H} \mid \text{ $H$ is non-compact} \}$ and $G/\mathcal{H}_{\infty} = \{gH \mid g \in G, H \in \mathcal{H}_{\infty}\}$.  
\begin{lemma}\label{lem:infiniteValence2}
If  $\Gamma $ is a Cayley-Abels graph for a proper pair $(G, \mathcal H)$, then there exists a $G$-equivariant bijection between $G/\mathcal{H}_{\infty}$ and $\mathcal{V}_\infty(\Gamma)$.
\end{lemma}

\begin{proof} 
By Theorem~\ref{thm:Cayley-Abels}, for any $H_i \in \mathcal{H}_{\infty}$ there exists a vertex $v_i \in \Gamma$ such that $G_{v_i} = H_i$.  Since $(G,\mathcal H)$ is proper, no two distinct subgroups of $\mathcal{H}_{\infty}$ are conjugate. It follows that  if $H_i$ and $H_j$ are distinct subgroups in $\mathcal H$, then $v_i$ and $v_j$ are in distinct $G$-orbits. 
By Theorem~\ref{thm:Cayley-Abels}, each $v_i$ has infinite degree in $\Gamma$;  if $w$ is a vertex of $\Gamma$ with $G$-stabilizer a conjugate of $H_i$ then $w$ and $v_i$ are in the same $G$-orbit; and any vertex in $V_\infty(\Gamma)$ has non-compact $G$-stabilizer.  Consider the $G$-map $\pi\colon G/\mathcal{H}_{\infty} \to \mathcal{V}_\infty$ given by $gH_i \mapsto gv_i$. Then
\begin{enumerate}
    \item $\pi$ is well defined since   $gH_i = fH_i$ if and only if $f^{-1}g \in H_i = G_{v_i}$ if and only if  $gv_i= fv_i$.
    \item $\pi$ is injective.
    Suppose $gH_i, fH_j \in G/\mathcal{H}_{\infty} $ and $gv_i = fv_j$. Then
    $v_i$ and $v_j$ are in the same $G$-orbit and hence $H_i=H_j$. Thus $gv_i = fv_i$ implies $gH_i = fH_i$.

    \item $\pi$ is surjective.
    Let $w \in \mathcal{V}_{\infty}$. 
    Then   $G_w$ is non-compact and therefore $G_w$ is conjugate to some $H_i \in \mathcal{H}_\infty$. Hence $v_i$ and $w$ are in the same $G$-orbit, say $w = gv_i$. Then $gH_i \in G/\mathcal{H}_{\infty}$ maps to $w$.\qedhere\end{enumerate}
\end{proof}

\begin{proof}[Proof of Theorem~\ref{thm:uniqueCAgraph}]
Let $(\mathcal{G}_1, \Lambda_1, \phi_1)$ and $(\mathcal{G}_2, \Lambda_2, \phi_2)$ be two compact generating graphs of $G$ relative to $\mathcal{H}_1$ and $\mathcal{H}_2$ respectively, and let $\Gamma_1(\mathcal{G}_1, \Lambda_1, \phi_1)$ and  $\Gamma_2(\mathcal{G}_2, \Lambda_2, \phi_2)$ be their corresponding Cayley-Abels graphs. By Lemma~\ref{lem:infiniteValence2}, there exists a $G$-equivariant bijection $\eta$ between $V_{\infty}(\Gamma_1)$ and $V_{\infty}(\Gamma_2)$. The result follows from Corollary~\ref{prop:EdgeAttachments}.  
\end{proof}

\section{Compact relative presentations}\label{sec:TopCharComplex}

A discrete group is finitely presented if and only if it acts cellularly, cocompactly and with finite vertex stabilizers on a simply connected space. This result is credited to ~\cite{Mac}, see~\cite[Corollary 8.11]{BrHa99}. The main results of this section, Theorem~\ref{prop:TopCharComplex} and Corollary~\ref{cor:TopcharCWcomplex}, generalize this fact for proper pairs $(G, \mathcal H)$. We introduce the notions of compact generalized presentation and  Cayley-Abels complex for pairs $(G, \mathcal H)$ in order to state our main result.  

\begin{definition}[Compact generalized presentation]
Let $G$ be a topological group and let $\mathcal{H}$ be a finite collection of open subgroups. A \emph{compact generalized presentation  of $G$ relative to $\mathcal{H}$} is a pair 
\begin{equation}\label{eq:relpr}
  \langle \  (  \mathcal{G}, \Lambda, \phi)\ |\  R \   \rangle
\end{equation} 
where $(  \mathcal{G}, \Lambda, \phi)$ is a compact generating graph of $G$ relative to $\mathcal{H}$, and $R \subseteq \pi_1(\mathcal{G},\Lambda) $ is a finite subset such that $ \nclose{R} = ker(\phi)$.
In particular, \[ 1 \to \nclose{R} \to \pi_1(\mathcal{G},\Lambda) \xrightarrow{\phi} G \to 1      \] is a short exact sequence and  $\phi$ is a continuous open surjective homomorphism.
\end{definition}
Compact generalized presentations have  been previously studied for TDLC groups~\cite{CaIWei16}.

\begin{definition}
    A topological group $G$ is said to be \emph{compactly presented relative to a finite collection of open subgroups $\mathcal{H}$} if there exists a compact generalized presentation relative to $\mathcal{H}$. 
\end{definition}

\begin{remark}[Compact generalized presentation $\Longrightarrow$ Compact open subgroup]\label{rem:GenGraph2} If  a topological group $G$ is compactly presented relative to a finite collection of open subgroups $\mathcal{H}$, then $G$ has a compact open subgroup or $\mathcal{H}=\{G\}$.  This is a consequence of Remark~\ref{rem:GenGraph}.
\end{remark}

The following are some motivating examples  of compact generalized presentations.  

\begin{example}
   \label{item:dicreteGroup} Let $G$ be a finitely presented group with presentation $\mathcal{P}= \langle X | R \rangle$. Then a compact generalized presentation of $G$ relative to the empty collection is given by $ \langle \  (  \mathcal{G}, \Lambda, \phi)\ |\  R \   \rangle$, where $(\mathcal{G}, \Lambda, \phi)$ is the graph of groups described in Example~\ref{exemp:fgGroup}.
\end{example}

\begin{example}
    Let $G$ be a discrete group finitely generated with respect to a finite collection of groups $\mathcal{H}= \{H_1, H_2\cdots H_n\}$. Let $X \subseteq G$ such that $(\bigcup\limits_{i=1}^{n}H_i) \cup X$ generates $G$. Let  $\langle X, \mathcal{H} | R \rangle$ be the relative presentation of $G$ with respect to $H$ as defined by Osin \cite{Os06}. Then a compact generalized presentation of $G$ relative to $H$ is given by $ \langle \  (  \mathcal{G}, \Lambda, \phi)\ |\  R \   \rangle$, where $(\mathcal{G}, \Lambda, \phi)$ is the graph of groups described  in Example~\ref{exemp:RelfgGroup}. 
\end{example}

\begin{example}\label{example:AmalgamatedFreeProductPresentation}
  Let $A \ast_C B$ be an amalgamated free product of topological  groups $A$ and $B$ over a common compact open subgroup with the  topology such that the  inclusions $A\hookrightarrow A\ast_C B$ and 
$B\hookrightarrow A\ast_C B$ are open topological embeddings, see   Proposition~\ref{prop:GTopology}.  Let $R \subset A\ast_C B$  satisfy the $C'(1/12)$ small cancellation condition,  let $N =\nclose{R}$ be the normal subgroup of $A\ast_CB$ generated by $R$, let $G =(A \ast_C B) /N$, and let $\phi\colon A\ast_CB\to G$ the quotient map.   
The small cancellation condition implies that the compositions $A\hookrightarrow A\ast_C B \to G$ and $B\hookrightarrow A\ast_C B \to G$ are injective, see~\cite[Ch. V. Theorem 11.2]{LySc01}. In particular,  $N\cap A$ is the trivial subgroup and therefore, since $A$ is open, it follows that $N$ is a discrete subgroup of $A\ast_CB$. In Example~\ref{examp:CompGenGraphAmalgamation}, we saw $(A\ast_C B, \phi)$ is compact generating graph of $G$ with respect to $\{A,B\}$, and therefore $\langle (A\ast_C B, \phi) \mid R \rangle$ is a compact generalized presentation of $G$ with respect to $\{A,B\}$, see Figure~\ref{fig:3}.

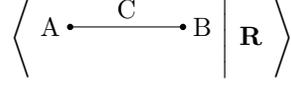
\begin{figure}
    \centering
    $\Braket{
    \begin{tikzpicture}
\draw (0,0) -- node[above]{C} (1.5,0) ;
\draw[black,thick, fill] (0,0) circle(0.8pt) node[left]{A};
\draw[black,thick, fill] (1.5,0) circle(0.8pt) node[right]{B};
\end{tikzpicture} \Bigg| \textbf{ R }
}$
    \caption{A compact generalized presentation  of the quotient of an amalgamated free product.}\label{fig:3}   
\end{figure}
\end{example}

\begin{definition}[Cayley-Abels complex]\label{def:RelCAComplex}
    Let $G$ be a topological group compactly generated relative to a finite collection $\mathcal{H}$ of open subgroups. A \emph{Cayley-Abels complex  of $G$ with respect to $\mathcal{H}$} is a discrete simply connected cocompact 2-dimensional $G$-complex with  1-skeleton a Cayley-Abels graph of $G$ with respect to $\mathcal{H}$.
\end{definition}

\begin{theorem}[Topological Characterization] \label{prop:TopCharComplex}      
Let $G$ be a topological group and let $\mathcal{H}$ be  a finite collection  of open subgroups. The following statements are equivalent:

\begin{enumerate}
    \item $G$ is compactly presented with respect to $\mathcal{H}$.
\item There exists a  Cayley-Abels complex of $G$ with respect to $\mathcal{H}$.
    \end{enumerate}
\end{theorem}

The following corollary  strengthens the conclusion of the previous theorem by imposing an  additional hypothesis; it is proven in Section~\ref{subsec:CPUnique}.

\begin{corollary}\label{cor:TopcharCWcomplex}
Let $(G,\mathcal{H})$ be a proper pair. Suppose there is a fine Cayley-Abels graph of $G$ with respect to $\mathcal{H}$. Then the following statements are equivalent:
\begin{enumerate}
    \item \label{item:CP} $G$ is compactly presented with respect to $\mathcal{H}$.
 \item \label{item:topComplex} Any Cayley-Abels graph of $G$ with respect to $\mathcal H$ is the 1-skeleton of a Cayley-Abels complex of $G$ with respect to $\mathcal H$.
\end{enumerate}
\end{corollary}

\begin{definition}[Compactly presented]\cite[Definition 8.A.1]{CoHa16}
A topological group $G$ is \emph{compactly presented} if $G$ has a group presentation $\langle S|R\rangle$ with $S$ compact and $R$ a set of words in $S$ with uniformly  bounded length. 
\end{definition}

The proof of the following corollary is implicit in the results of~\cite{CoHa16} and ~\cite[Section 5.8]{CaIWei16}. A proof in the class of TDLC groups is discussed  in~\cite[Corollary 3.5]{ACCMP}. We discuss the argument in Section~\ref{subsec:cor5.5}. 

\begin{corollary}\label{prop:CompactPres}
The following statements are equivalent for any topological group $G$.
\begin{enumerate}
    \item $G$ admits a compact generalized presentation relative to the empty collection.
    \item $G$ has a compact open subgroup and is  compactly presented.
\end{enumerate}
\end{corollary}

As a consequence, we also recover the following corollary from~\cite[Proposition 3.6]{ACCMP}, the proof of which relied on a result from~\cite[Proposition 3.4]{Cook}.

\begin{corollary}\cite{Cook}
    A TDLC group $G$ is compactly presented if and only if there exists a simply connected cellular $G$-complex $X$ with compact open cell stabilizers, finitely many $G$-orbits of cells of dimension at most 2, and such that elements of $G$ fixing a cell setwise fix it pointwise (no inversions).
\end{corollary}
The rest of the section is subdivided into three subsections covering the proofs of the Theorem~\ref{prop:TopCharComplex}, Corollary~\ref{cor:TopcharCWcomplex}, and Corollary~\ref{prop:CompactPres} respectively.

\subsection{Proof of the Theorem~\ref{prop:TopCharComplex}}
The proof follows directly from Proposition~\ref{prop:RelPresComplex} and Proposition~\ref{prop:ComPresTop} below.

Proposition~\ref{prop:RelPresComplex}   is a standard construction for discrete groups, see~\cite[Lemma 8.9]{BrHa99}, and it appears as~\cite[Theorem 7.1]{ACCMP} for TDLC groups that split as an amalgamated free  product over a compact open subgroup.  

\begin{proposition}\label{prop:RelPresComplex}
Let $G$ be a topological group compactly presented relative to a finite collection $\mathcal{H}$ of open subgroups. Then there exists a Cayley-Abels complex $X$ of $G$ relative to $\mathcal{H}$ such that no distinct 2-cells of $X$ have the same boundary. 
\end{proposition}
\begin{proof}
 Let $\mathcal{P}= \langle \  (  \mathcal{G}, \Lambda, \phi) \mid   R \   \rangle$ be a compact generalized presentation of $G$ relative to $\mathcal{H}$, and let $\mathcal{T}$ be the corresponding Bass-Serre tree. 
 Consider a $G$-complex $X$ constructed as follows. 
   Let 1-skeleton $X^{(1)}$ be the Cayley-Abels graph $ \Gamma = \mathcal{T}/ \ker(\phi)$, and let $\rho \colon \mathcal{T} \to \Gamma$ be the natural quotient map. Suppose $x \in \Gamma$ be any vertex, then by Proposition~\ref{rem:RelativeCayley}, there is a group isomorphism from $\ker(\phi) \to \pi_1(\Gamma, x)$ given by $g \mapsto [\gamma_g]$, where $\gamma_g$ is an combinatorial closed path in $\Gamma$ based at $x$. For $g\in G$ and $h\in \ker(\phi)$,  let $g\gamma_h$ be the translated closed path without an initial point, i.e., these are cellular maps from $S^1 \to X$. Consider the $G$-set $\Omega=\{g\gamma_{r}\mid r\in R, g\in G\}$ of closed paths in $X^{(1)}$.
 
  Let $X$ be the $G$-complex obtained by attaching a 2-cell to $X^{(1)}$ for every closed path in $\Omega$. In particular, the pointwise $G$-stabilizer  of a 2-cell of $X$ coincides with the pointwise $G$-stabilizer of its boundary path. By Proposition~\ref{rem:RelativeCayley}, the $G$-action is discrete on $X$.
  Observe that the natural isomorphism from $\ker(\phi)$ to  $\pi_1(X^{(1)}, \rho(x_0))$ implies that $X$ is simply connected and since $G$ acts cocompactly on $X^{(1)}$ and $R$ is finite, $X$ is a cocompact $G$-complex. Note that the definition of $\Omega$, implies that no two distinct 2-cells of $X$ have the same boundary.
\end{proof}

\begin{example}\label{exp:Amalgamation}
Let $G=A\ast_CB/\nclose{r^m}$, where $A$ and $B$ are locally compact  groups, $C$ is a common compact open subgroup, and $r\in G$ such that (the symmetric set induced by)  $r^m$ satisfies the $C'(1/12)$ small cancellation condition. Consider a compact generalized presentation $\langle \  (  A\ast_CB, \phi) \mid  r^m \   \rangle$ of $G$ relative to $\{A,B\}$ as in Example~\ref{example:AmalgamatedFreeProductPresentation}. Let $\mathcal{T}$ be the Bass-Serre tree corresponding to $A\ast_CB$, and let $\phi \colon \pi_1(\mathcal{G}, \Lambda) \to G$ be the natural quotient map. 
Then the corresponding Cayley-Abels complex $X$ is a simply connected 2-complex with 1-skeleton given by $\mathcal{T}/\ker(\phi)$. The $G$-complex $X$ has a single orbit of 2-cells. Let $D$ be a 2-cell such that $\phi(r) \subseteq G^D$, where $G^D$ is the setwise stabilizer of $D$. Observe that $G^D$ is a subgroup of $\langle r \rangle$, and the $G^D$-translates of the  path in $X^{(1)}$ corresponding to $r$ cover   $\partial D$. 
\end{example}

\begin{proposition}\label{prop:ComPresTop}
Let $G$ be a topological group with compact generating graph  $(\mathcal{G}, \Lambda, \phi )$  relative to a finite collection of open subgroups $\mathcal{H}$, and let $\Gamma$ be the corresponding Cayley-Abels graph.
If  $\Gamma$ is the 1-skeleton of a discrete simply connected cocompact $G$-complex $X$, then $G$ has a compact generalized relative presentation relative to $\mathcal{H}$ with generating graph $(\mathcal{G}, \Lambda, \phi )$.
\end{proposition}
\begin{proof}
Since $X$ is simply connected, if there are no 2-cells in $X$, $\ker(\phi)$ is trival and there is compact generalized presentation $\langle \  (  \mathcal{G}, \Lambda, \phi) \ \rangle$. Suppose that $\ker(\phi)$ is not trivial.  Since $G$ acts cocompactly on $X$, there exists a finite collection $\{\Delta_1, \Delta_2, \cdots \Delta_k\}$ of $G$-orbit representatives of 2-cells in $X$. Let $\{\gamma_1, \gamma_2, \cdots, \gamma_k\}$ be their corresponding boundary paths. Let $x_0$ be a fixed vertex in $\mathcal{T}$ and for any $\gamma_i$, let $\alpha_i$ be a path in 1-skeleton of $X$ from $\rho(x_0)$ to some fixed vertex of $\gamma_i$. Lift the concatenated paths $\alpha_i \ast \gamma_i$ to $\widetilde{\alpha_i \ast\gamma_i}$ in $\mathcal{T}$ starting at $x_0$. Then there exist $r_i \in \ker(\phi)$ such that $r_i.x_0$ is the endpoint of $\widetilde{\alpha_i \ast\gamma_i}$. As a consequence of Proposition~\ref{rem:RelativeCayley}, $R = \{r_1, r_2, \cdots,r_k\}$ is a finite set normally generating $\ker(\phi)$.
Hence we have a compact relative presentation $\langle \  (  \mathcal{G}, \Lambda, \phi)\mid   R \   \rangle$ of $G$.
\end{proof}

\subsection{Proof of the Corollary~\ref{cor:TopcharCWcomplex} }\label{subsec:CPUnique}

We  use the following definition of large-scale simply connectedness introduced in~\cite[Definition 1.3]{Salle}.

\begin{definition}\label{def:OmeganGamma}
Let $\Gamma$  be a connected graph. For $k \in \N$, define a 2-complex $ \Omega_k(\Gamma)$ with 1-skeleton $\Gamma$ and 2-cells as $m$-gons for any simple loop of length $m$ up to cyclic permutation, where $0 \leq m \leq k$.
\end{definition}

\begin{definition}\cite[Definition 1.3]{Salle}
A graph $\Gamma$ is said to be $k$-simply connected if $\Omega_k(\Gamma)$ is simply connected. If there exists such a $k$, then we shall say that $\Gamma$ is large-scale simply connected.
\end{definition}

\begin{remark}\label{rem:LargeSCinvariant}We note the following observations:
\begin{enumerate}
 \item If $\Gamma$ is a $G$-graph, then $\Omega_k(\Gamma)$ is a $G$-complex.
\item Large-scale simple connectedness is preserved by quasi-isometry, see~\cite[Theorem 2.2]{Salle}.
 \item 
 If a graph $\Gamma$ is a 1-skeleton of a simply connected cocompact $G$-complex, then $\Gamma$ is large-scale simply connected.
\end{enumerate}
\end{remark}

\begin{proof}[Proof of the Corollary~\ref{cor:TopcharCWcomplex}]
\eqref{item:topComplex} implies \eqref{item:CP} is direct from Theorem~\ref{prop:TopCharComplex}. 
Conversely, let $\Gamma$ be any Cayley-Abels graph of $G$ with respect to $\mathcal H$. Since $G$ is compactly presented, Theorem~\ref{prop:TopCharComplex} implies there exists a Cayley-Abels graph $\Gamma_1$ that is large-scale simply connected, see Remark~\ref{rem:LargeSCinvariant}. Since there exists a fine Cayley-Abels graph of $G$ with respect to $\mathcal{H}$,   Corollary~\ref{cor:QI-Fineness} implies that  $\Gamma$ is large-scale simply connected and fine. Let $k>0$ such that $\Omega_k(\Gamma)$ is simply-connected.  Let $\mathcal{E}$ be a finite set of $G$-orbit representatives of edges in $\Gamma$. Since $\Gamma$ is fine, for any $e \in \mathcal{E}$, there exists finitely many circuits of length at most $k$ containing $e$. By construction, $\Omega_k(\Gamma)$ has finitely many 2-cells up to the $G$-action. Hence $\Omega_k(\Gamma)$ is a discrete simply connected cocompact $G$-complex with 1-skeleton $\Gamma$.
\end{proof}

\subsection{Proof of Corollary~\ref{prop:CompactPres}}\label{subsec:cor5.5}
We   use the following proposition for the proof.
\begin{proposition}\cite[Prop. 8.A.10]{CoHa16}\label{prop:ConShortExact}
Let $1 \to N \to G \to Q \to 1$ be a short exact sequence of locally compact groups  and continuous homomorphisms.
\begin{enumerate}

\item  Assume that $G$ is compactly presented and that $N$ is compactly generated as a normal subgroup of $G$. Then $Q$ is compactly presented.
\item Assume that $G$ is compactly generated and that $Q$ is compactly presented. Then $N$ is compactly generated as a normal subgroup of $G$.
\item If $N$ and $Q$ are compactly presented, then so is $G$.
\end{enumerate}

\end{proposition}

\begin{proof}[Proof of Corollary~\ref{prop:CompactPres}]
Suppose $\langle \  (  \mathcal{G}, \Lambda, \phi)\mid  R \   \rangle$ is a compact generalized presentation of $G$ relative to the empty collection. By Remark~\ref{rem:GenGraph2}, $G$ has a compact open subgroup. On the other hand, all vertex and edge groups of  $(  \mathcal{G}, \Lambda, \phi)$ are compact, and hence $\pi_1(\mathcal{G}, \Lambda)$ has a  compact generating set $S$ consisting of the union of all vertex groups and a finite set of elements (stable letters for HNN-extensions). Then a group presentation of $\pi_1(\mathcal{G}, \Lambda)$ with generating set $S$  is obtained by considering all relations given by the multiplication tables of the vertex groups and the HNN relations; note that all relations have length at most four. Since $R$ is a finite set, we have a bounded presentation for $G$. 
Conversely, assume $G$ has a bounded presentation $\langle S|R\rangle$ and a compact open subgroup. Theorem~\ref{prop:TopCharGraph} implies that   $G$ admits a compact generating graph $(  \mathcal{G}, \Lambda, \phi)$ relative to the empty collection. Since $\pi_1(\mathcal{G}, \Lambda)$ is compactly generated and $G$ is compactly presented, Proposition~\ref{prop:ConShortExact} implies that $\ker(\phi)$ is compactly generated as a normal subgroup. By Proposition~\ref{rem:RelativeCayley}, $\ker(\phi)$ is discrete and therefore finitely generated as a normal subgroup.
\end{proof}
\begin{remark}
 Corollary~\ref{prop:CompactPres} can also be obtained as a consequence of Theorem~\ref{prop:TopCharComplex} and \cite[Theorem 8.10]{BrHa99}
\end{remark}

\section{Relatively compactly presented groups}\label{Sec:Osin2}

In this section we prove two results whose restriction to discrete groups were proven by Osin in~\cite[Theorem 2.40 and Theorem 1.1]{Os06}. 

\begin{theorem} \label{prop:FinRelPresHisCPimpliesGisCP} \label{prop:FiniteRelPresGisCGimpliesHisCG}
Let $G$ be a topological group compactly presented relative to a finite collection $\mathcal{H}$ of open subgroups.
\begin{enumerate}
    \item \label{item:CP1} If each $H \in \mathcal{H}$ is compactly presented then $G$ is compactly presented.
    \item \label{item:CP2} If $G$ is compactly generated, then each $H \in \mathcal{H}$ is compactly generated. 
\end{enumerate}
\end{theorem}

\begin{remark}
Note that the statements of Theorem~\ref{prop:FinRelPresHisCPimpliesGisCP}  are trivial if $G$ has no compact open subgroup since in this case  Remark~\ref{rem:GenGraph2} implies that $\mathcal{H}=\{G\}$.  
\end{remark}

The section consists of three parts. The first subsection discusses  the proof of Theorem~\ref{prop:FinRelPresHisCPimpliesGisCP}\eqref{item:CP1}, then the other two sections contain the proof of Theorem~\ref{prop:FinRelPresHisCPimpliesGisCP}\eqref{item:CP2}.

\subsection{Normal forms}\label{subsec:NormalForms}
Let us recall the notion of normal form for  amalgamated free products and HNN extensions. For details we refer the reader to~\cite{LySc01}. 

\subsubsection*{Normal form for Amalgamated free product} 
Let $G = A \ast_CB$ be an amalgamated free product. 
Choose a system of representatives $T_A$ of right cosets of $C$ in $A$ and a system of
representatives $T_B$ of right cosets of $C$ in $B$.  Assume that 1 represents the coset of $C$ in $A$ and $B$. A \emph{normal} form in the amalgamated free product $A\ast_C B$ is a sequence $x_0x_1 \cdots x_n$ such that 
\begin{enumerate}
    \item $x_0 \in C$.
    \item $x_i \in T_A \setminus \{1\}$ or $x_i \in T_B \setminus \{1\}$. for $i \geq 1$, and consecutive terms $x_i$ and $x_{i+1}$ lie in distinct system of representatives.
\end{enumerate}

\subsubsection*{Normal form for HNN extension} 
Let $A$  be a group, let $B$ and $C$ be subgroups of $A$, and let $\alpha\colon C \to B$ be an isomorphism. Suppose $G$ is the group defined as the HNN extension
\[G = A\ast_{\alpha}= \langle A,t \mid t^{-1}ct=\alpha(c), c\in C \rangle. \] 
Choose a system of representatives $T_C$ of right cosets of $C$ in $A$ and a system of
representatives $T_B$ of right cosets of $B$ in $A$.  Assume that 1 represents the coset of $C$ and $B$. A \emph{normal form} in an HNN extension is a 
$g_0t^{\epsilon}_1g_1 \cdots t^{\epsilon_n}g_n$ such that 

\begin{enumerate}
    \item $g_0$ is an arbitrary element of $A$,
    \item if $\epsilon_i = -1$, then $g_i \in T_B$,
    \item if $\epsilon_i = 1$, then $g_i \in T_C$,
    \item there is no consecutive subsequence $t^{-1}1t$.
\end{enumerate}

\begin{theorem}\cite{LySc01}
Suppose $G=A\ast_C B$ or $G=A\ast_\alpha$. 
Every element $g \in G$ can be uniquely written as the product of sequence in   normal form. 
\end{theorem}

\subsection{Proof of Theorem~\ref{prop:FinRelPresHisCPimpliesGisCP}\eqref{item:CP1}}\label{Sec:Osin1}

\begin{lemma} \label{lem:GisCP}
If $(\mathcal{G}, \Lambda)$ is a finite  graph of topological groups with compactly presented vertex groups and compact edge groups, then $\pi_1(\mathcal{G},\Lambda)$ is compactly presented.
\end{lemma}	
\begin{proof}
It is enough to prove the result for the graph of groups corresponding to the amalgamated free products  and HNN extension. The general case follows by induction. Let $\Lambda$ be a single edge with compactly presented vertex groups $A$ and $B$ and compact edge group $C$. Let $i_A \colon C \to A $ and $i_B \colon C \to B $ be open topological embeddings. A presentation of $\pi_1(\mathcal{G}, \Lambda ) \simeq A\ast_CB$ is given by $\langle A, B \mid i_A(c)i_B(c)^{-1} \text{ for } c \in C \rangle$. Let $\langle S_A \mid R_A\rangle$ and $\langle S_B \mid R_B\rangle$ be bounded presentations of $A$ and $B$ over compact sets $S_A$ and $S_B$ respectively such that $C \subseteq S_A$ and $C\subseteq S_B$. Then   $\langle S_A \cup S_B \mid i_A(c)i_B(c)^{-1} \text{ for } c \in C \rangle$ is a bounded presentation of $G$ over a compact set, and hence $G$ is compactly presented. For HNN extension, let $\Lambda$ be an edge loop with compactly presented vertex group $A$ and compact edge group $C$. Let $i_1\colon C \to A$ and $i_2 \colon C \to A$ be open topological embeddings. Let $\langle S \mid R\rangle$ be a bounded presentation of $A$ over compact set $S$ such that $C \subseteq S$.
Let $y$ be a symbol not in $S$. Then a presentation of $\pi_1(\mathcal{G},\Lambda)$ is given by $\langle S, y\mid R, i_1(c)^{-1} yi_2(c)y^{-1} \text{ for } c \in C \rangle$. This is a bounded presentation over a compact set. Hence $\pi_1(\mathcal{G},\Lambda)$ is compactly presented.
\end{proof}

\begin{proof}[Proof of Theorem \ref{prop:FinRelPresHisCPimpliesGisCP}\eqref{item:CP1}] 
Suppose each $H\in \mathcal{H}$ is compactly presented. Let 
$\langle \  (  \mathcal{G}, \Lambda, \phi) \mid  R \   \rangle$ be a compact generalized presentation of $G$ relative to $\mathcal{H}$, and let $\widetilde{G} = \pi_1(\mathcal{G}, \Lambda)$ be the fundamental group of the graph of groups $(\mathcal{G}, \Lambda )$. Since $(\mathcal{G}, \Lambda)$ is a finite graph of topological groups with compactly presented vertex groups and compact edge groups, by Lemma~\ref{lem:GisCP} , $\widetilde{G}$ is compactly presented. Consider the short exact sequence $1 \to \ker(\phi) \to \widetilde{G} \to G \to 1$. Since $R$ is finite, $\ker(\phi)$ is compactly generated as a normal subgroup. By Proposition~\ref{prop:ConShortExact}, $G$ is compactly presented.
\end{proof}

\subsection{
Compactly generated topological graphs of groups.
}

In this part, we prove the following proposition which is a particular case of Theorem~\ref{prop:FiniteRelPresGisCGimpliesHisCG}\eqref{item:CP2}.

\begin{proposition}\label{prop:VertexGroupsAreCG}
Let $G$ be the fundamental group of a finite graph of topological groups $(\mathcal{G}, \Lambda)$ such that edge groups are compactly generated. If $G$ is compactly generated, then vertex groups of $(\mathcal{G}, \Lambda)$  are compactly generated.
\end{proposition}

This proposition can also be stated as follows:
\begin{corollary}\label{cor:IlariaGeeralisation}
Let $G$ be a topological group acting discretely, cocompactly and without inversions on a tree such that the edge stabilizers are compactly generated. If $G$ is compactly generated, then the vertex stabilizers are compactly generated.
\end{corollary}

Corollary~\ref{cor:IlariaGeeralisation} in the case that $G$ is a TDLC group and the edge stabilizers are compact is a result of Castellano~\cite[Proposition 4.1]{castellano_2020}.  The proof of Proposition~\ref{prop:VertexGroupsAreCG} follows by induction on the number of edges of $\Lambda$, so it reduces to proving the result for amalgamated free products and HNN extensions.

\begin{lemma}
Let $G=A\ast_C B$ be a topological group such that $G$ and $C$ are compactly generated, and $C$ is an open subgroup containing a compact open subgroup. Then $A$ and $B$ are  compactly generated. 
\end{lemma}
\begin{proof}
Let $U$ be a compact open subgroup of $C$. Let $C_A$ and $C_B$ denote the copies of $C$ in $A$ and $B$ respectively. Since $G$ and $C$ are compactly generated,   there are  finite subsets $X\subset G$ and $Y\subset C$ such that $G=\langle X\cup U\rangle$ and $C=\langle Y \cup U \rangle$. For each element of $X$ choose a normal form; and let $Z \subset A$ consists of all $a\in A$ such that $a$ appears in a chosen normal form of an element of $X$. Observe that $Z$ is a finite set. 
 We claim that $A=\langle Y\cup Z \cup U\rangle$ and hence $A$ is compactly generated. Let $A'$ be the subgroup of $A$ generated by $Y\cup Z \cup U$. Since $C_A=\langle Y\cup U \rangle$, it follows that $C_A\leq A'$. Let  $\psi\colon A'\ast_C B \to G$ the morphism induced by the inclusions $A'\hookrightarrow A \hookrightarrow G$ and   $B\hookrightarrow G$. Then $\psi$ is surjective since its image contains $\psi(X\cup U)=X\cup U$ which generates $A\ast_C B$. Note that $\psi$ preserves the length of normal forms and $\psi(B)=B$; therefore surjectivity of $\psi$ implies that $\psi(A')=A$. By symmetry, $B$ is also compactly generated.
\end{proof}

\begin{lemma}
Let $A$ be a topological group, let $B$ and $C$ be  open subgroups of $A$ containing compact open subgroups and let $\alpha\colon C \to B$ be an isomorphism of topological groups. Let $G$ be the TDLC group defined as the HNN extension
\[G = A\ast_{\alpha}= \langle A,t \mid t^{-1}ct=\alpha(c), c\in C \rangle. \]
If $G$ and $C$ are compactly generated, then $A$ is compactly generated.
\end{lemma}
\begin{proof}
Let $U$ be a compact open subgroup in $C \cap B$. Since $G$, $C$, and $B$ are compactly generated, there are  finite subsets $X\subset G$, $Y\subset C$, and $W \subset B$ such that $G=\langle X\cup U\rangle$, $C=\langle Y \cup U \rangle$, and $B=\langle W \cup U \rangle$. For each element of $X$ choose a normal form; and let $Z \subset A$ consists of all $a\in A$ such that $a$ appears in a chosen normal form of an element of $X$. Observe that $Z$ is a finite set. 
 We claim that $A=\langle W \cup Y\cup Z \cup U\rangle$ and hence $A$ is compactly generated. Let $A'$ be the subgroup of $A$ generated by $W \cup Y\cup Z \cup U$. 
 Since $B= \langle W\cup U \rangle $ and $C=\langle Y\cup U \rangle$, it follows that $B \leq A'$ and $C\leq A'$. Consider the group $G' = A'\ast_{\alpha}$ and the morphism $\psi\colon G' \to G$ induced by the inclusions $A'\hookrightarrow A \hookrightarrow G$. Then $\psi$ is surjective since its image contains $\psi(X\cup U)=X\cup U$ which generates $G$. Note that $\psi$ preserves the length of normal forms; therefore surjectivity of $\psi$ implies that $\psi(A')=A$.
\end{proof}

\subsection{Proof of the Theorem~\ref{prop:FiniteRelPresGisCGimpliesHisCG}\eqref{item:CP2}}
\begin{definition}[Link]
 Let $X$ be a simplicial 2-complex and let $\sigma$ be a cell in $X$. The \emph{star} $\st(\sigma)$ of $\sigma$ is defined as the set of all closed cells of $X$ incident to $\sigma$.  
 The \emph{link} of $\sigma$ in $X$, denote as $\link_X(\sigma)$,  is the subcomplex of $X$ defined as the set of  closed cells of $\st(\sigma)$ that do not intersect $\sigma$. 
 If $\sigma$ is a 0-cell, $\link_X(\sigma)$ can also be interpreted as 
 a graph $\Gamma = \link_X(\sigma)$ defined with the set of vertices \[V(\Gamma) = \{e \mid e \text{ is a 1-cell in }\st(\sigma) \text{ adjacent to } \sigma \}\] and the set of edges $E(\Gamma) = \{(e_1,e_2)\}$, where $e_1,e_2$ are pair of edges in $\st(\sigma)$ adjacent to $\sigma$ contained in the boundary of same 2-cell in $\st(\sigma)$.
\end{definition}

\begin{remark}\label{rem:barycentric}  We note the following facts without proof.
\begin{enumerate}
\item Let $X$ be a 2-dimensional $G$-complex with cellular and cocompact action of $G$. Then $G$ acts cocompactly on the barycentric subdivision of $X$.  
\item Let $X$ be a 2-dimensional $G$-complex. After sufficient subdivisions, $X$ can be considered a simplicial $G$-complex. In particular, for any cell $\sigma$ in $X$, $\link_X(\sigma)$ is well-defined and we can consider the $G_{\sigma}$-action on the $\link_X(\sigma)$.
\item Let $X$ be a 2-dimensional simplicial $G$-complex and let $X'$ be the barycentric subdivision of $X$. For any 0-cell $\sigma \in X$, there is an $G_{\sigma}$-equivariant isomorphism  $\link_X(\sigma) \to \link_{X'}(\sigma)$.
\end{enumerate}	
\end{remark}

\begin{lemma}\label{lem:CocompactAction}
Let $X$ be a 2-dimensional cocompact simplicial $G$-complex without inversions. Let $\sigma \in X$ be a 0-cell. Then $G_{\sigma}$ acts cocompactly on $\link_X(\sigma)$. 
\end{lemma}
\begin{proof}
Let $X'$ be the barycentric subdivision of $X$ with the induced cellular $G$-action. Then $G$ acts on $X'$ without inversions and, by Remark~\ref{rem:barycentric}, the action is cocompact. 
Let $st(\sigma)$ be the set of closed cells of $X'$ incident to $\sigma$. Since the $G$-action on $X'$ is cocompact, there exists finitely many cells $\{\tau_i\}$ in  $st(\sigma)$ such that for any $\tau \in st(\sigma) $ there exists $g_{\tau} \in G$ such that $g_{\tau}\tau = \tau_i$ for some $i$. By the definition of induced action on $X'$, $g_{\tau}\sigma = \sigma$. Therefore, $g_{\tau} \in G_{\sigma}$. Hence $G_{\sigma}$ acts cocompactly on $st(\sigma)$ and thus on $\link_{X'}(\sigma)$. 
\end{proof}

The proof of the following lemma uses disc diagrams. The  remark below recalls the notion of disc diagram and makes some observations.

\begin{remark}[Boundary cycle of Disc diagrams] \label{rem:DiscDiag}
A \emph{disc diagram} $D$ is a nonempty contractible finite 2-complex that has an specified embedding 
in a 2-sphere. If $D$  consists of a single 0-cell then it is called trivial.
For a non-trivial disc diagram $D$, roughly speaking, the boundary cycle $S^1\to D$ is the closed path (up to orientation) 
around the component of the complement of $D$ in the sphere. In particular, 
a disc diagram $D$ is homeomorphic to a disc if and only if its boundary cycle  $\partial D \to D$
is an embedding of the circle. For background on disc diagrams we refer the reader to~\cite[Section 2]{mccammond_wise_2002}.
\end{remark}

\begin{lemma}\label{lem:ConnectedComponentCorrespondence}
Let $X$ be a simply connected 2-dimensional simplicial $G$-complex and let $v$ be a 0-cell. Let $\mathcal{K}$ be the set of connected components of $\link_X(v)$ and $\Omega$ be the set of
connected components of $X\setminus\{gv \mid g \in G\}$ whose closure contains $v$.
Then there is a $G_v$-equivariant bijection between $\mathcal{K}$ and $\Omega$.
\end{lemma}

\begin{proof}
Since $\link_X(v)$ is a subspace of $X\setminus\{gv \mid g \in G\}$, there is a natural $G_v$-equivariant inclusion $\eta \colon \mathcal{K} \to \Omega$. We claim that $\eta$ is a bijection.
 Let $k_1,k_2 \in \mathcal{K}$ and $x, y$ be vertices of $X$ in $k_1$ and $k_2$ respectively. Suppose $\eta(k_1) =\eta(k_2)$.
 Then there exists a path $p$ in $X\setminus\{gv \mid g \in G\}$ between $x,y$. Since $k_1,k_2 \in \mathcal{K}$, there are paths $p_1,p_2$ in $X$ joining $v, x$ and $y, v$ respectively.
 Choose $p_1,p_2,p$ of minimal combinatorial length.  Let $\gamma$ be the loop in $X$ obtained by concatenating $p_1, p, p_2$ and observe that minimality  implies that $\gamma$ is an embedding in $X$. Since $X$ is simply connected, there exists a disc diagram $D$ and a simplicial map $\phi\colon D \to X$ such that  $\gamma = \phi \circ \gamma'$, where $\gamma' \colon S^1 \to D$ is the boundary cycle of $D$, see for example~\cite[Lemma 2.17]{mccammond_wise_2002}. Since $\gamma$ is an embedding, we have that $D$ is homeomorphic to a disc, see Remark~\ref{rem:DiscDiag}. Let  $w$ be the vertex on the boundary of $D$ mapping to $v$ and observe that $\link_D(w)$ is homeomorphic to an interval. Then $\phi$ induces a path   $\link_D(w)\to \link_X(v)$ between $x$ and $y$ in $\link_X(v)$. Thus $k_1 = k_2$ and $\eta$ is injective.   
 
Let $\omega \in \Omega$. Then $\omega$ contains the interior of a cell $\sigma$ incident to $v$. Let $k$ be the connected component of $\link_X(v)$ containing $\link_X(v) \cap \sigma$. Then $k$ maps to $\omega$, and hence $\eta$ is surjective.  
\end{proof}

\begin{proposition}\label{lem:stabLinkisCG}
Let $X$ be a simply connected 2-dimensional simplicial $G$-complex with discrete cocompact $G$-action such that every cell of dimension greater than 0 has compact stabilizer. Then for any vertex $v$, the setwise $G_v$-stabilizer of  any connected component of $\link_X(v)$ is compactly generated. 
\end{proposition}
\begin{proof}
Let $\Delta$ be a  connected component of $\link_X(v)$ and let $K$ be its setwise $G_v$-stabilizer. Since $X$ is a discrete $G$-complex and stabilizers of 1-cells and 2-cells are compact, $\Delta$ is a connected discrete $K$-graph with compact cell stabilizers. By Lemma~\ref{lem:CocompactAction}, we also have that  $\Delta$ is a cocompact $K$-graph. Hence by Theorem~\ref{prop:TopCharGraph}, $K$ is compactly generated.
\end{proof}

\begin{proof}[Proof of the Theorem~\ref{prop:FiniteRelPresGisCGimpliesHisCG}\eqref{item:CP2}]
Let $\mathcal{P} = \langle \  (  \mathcal{G}, \Lambda, \phi) \mid  R \   \rangle$ be a compact relative generalized presentation of $G$ with  respect to the collection $\mathcal{H}$ of open subgroups, and let $X$ be the barycentric subdivision of the corresponding Cayley-Abels complex. In particular, the $G$-action on $X$ has no inversions, links of cells are well defined, and every cell of dimension greater than 0 has compact stabilizer. 

 Let $H \in \mathcal{H}$  and let $v_H$ denote a 0-cell of $X$ such that $G_{v_H} = H$, note that such vertex exists by Proposition~\ref{prop:RelPresComplex}. The proof is divided into  two cases: 
  
  \textbf{Case 1:} $v_H$ is not a cut-point in $X$. 
 
  By Lemma~\ref{lem:ConnectedComponentCorrespondence}, $\link_{X}(v_H)$ is connected.
  In this case $H$ acts on $\link_{X}(v_H)$  discretely, cocompactly and with compact stabilizers.  By Proposition~\ref{lem:stabLinkisCG}, $H$ is compactly generated.
  
  \textbf{Case 2:} $v_H$ is a cut-point in $X$. 
   
  Observe that for any $g \in G$, the 0-cell $gv_H$  is a cut-point in $X$. Let $\Omega$ be the set of connected components of the set $X \setminus \{gv_H \mid g \in G\}$. For any $\Delta \in \Omega$, denote $\overline{\Delta}$ be the closure of $\Delta$ in $X$.
Construct a tree $T$ with vertex set
$V(T)= \{gv_H \mid g \in G\} \cup \{v_{\Delta} \mid \Delta \in \Omega\} $ and edge set $E(T) = \{\{gv_H,v_{\Delta}\}\mid gv_H \in \overline{\Delta}, g \in G, \Delta \in \Omega\}$. Observe that   $T$ is  a tree with a natural $G$-action and $G_{v_H}= H$. We claim that the edge $G$-stabilizers in $T$ are compactly generated. Without loss of generality, consider an edge $e = \{v_H,v_{\Delta}\}$ incident to vertex $v_H$ in $T$. Then  $G_e$ is the setwise $H$-stabilizer  of $\Delta$. By Lemma~\ref{lem:ConnectedComponentCorrespondence}, $G_e$ is the setwise $H$-stabilizer of a connected component of $\link_{X}(v_H)$. Thus Proposition~\ref{lem:stabLinkisCG}
implies $H_{\Delta}$ is compactly generated. Therefore $G$-stabilizers of edges of $T$   are compactly generated; and since $G$ is compactly generated, Proposition~\ref{prop:VertexGroupsAreCG} implies vertex stabilizers are compactly generated. In particular, $H$ is compactly generated.
\end{proof}

\section{Relatively hyperbolic groups}\label{sec:relhp}

In this section, we generalize the notion of relatively hyperbolic group for proper pairs $(G, \mathcal H)$. Our definition extends Bowditch's approach to relative hyperbolicity for discrete groups~\cite{Bo12}. 
The main result of the section is that relatively  hyperbolic groups  admit compact relative presentations, see Theorem~\ref{prop:RelHypimpliesFinitePres}.

\begin{definition}\label{def:RelHyp}
Let $(G,\mathcal H)$ be a proper pair.  The group $G$ is \emph{relatively hyperbolic} 
with respect to $\mathcal{H}$ if there exists a Cayley-Abels graph $\Gamma$ of $G$ with respect to $\mathcal{H}$ which is fine and hyperbolic.
\end{definition}

 \begin{remark}\label{rem:relhyp}
  Let $(G, \mathcal{H})$ be a proper pair. If $G$   is  hyperbolic relative to   $\mathcal H$, then: 
 \begin{enumerate}
\item the group $G$ is compactly generated relative to $\mathcal H$ by Theorem~\ref{prop:TopCharGraph};  and 
\item Cayley-Abels graphs of $G$ with respect to $\mathcal{H}$ are fine and hyperbolic by  Theorem~\ref{cor:QI-Fineness}.
\end{enumerate}
 \end{remark}
 
\begin{example}
Let $G$ be the fundamental group  of a finite graph of topological groups $(\mathcal{G}, \Lambda)$ with compact open edge groups. Then  $G$ is   hyperbolic relative to the collection $\mathcal{H}$ of vertex groups $\mathcal{G}_v$. Indeed,  the Bass-Serre tree of $(\mathcal{G}, \Lambda)$ is a   Cayley-Abels graph of $G$ with respect to $\mathcal H$ which is hyperbolic and fine. 
\end{example}

\begin{theorem}\label{prop:RelHypimpliesFinitePres}
Let $(G,\mathcal H)$ be a proper pair. If $G$ is hyperbolic relative to $\mathcal{H}$, then $G$ is compactly presented relative to $\mathcal{H}$.
\end{theorem}

Small cancellation quotients of free products are a source of relatively hyperbolic groups in the discrete case~\cite[Example(II) Page 4]{Os06}.  Proposition~\ref{lem:AmalgamationHyperbolic}  generalizes this construction.  For background on small cancellation quotients of amalgamated free products we refer the reader to the book by Lyndon and Schupp~\cite[Ch.V]{LySc01}.

\begin{proposition}\label{lem:AmalgamationHyperbolic}
 Let $A \ast_C B$ be a topological group that splits as an amalgamated free product  over a common compact open subgroup $C$. If $R \subseteq A\ast_CB$ is  symmetrized set satisfying $C'(\lambda)$ condition and   $G = (A\ast_CB)/\nclose{R}$ then: 
 \begin{enumerate}
 \item  For the compact generalized presentation
 $\langle (  A\ast_CB, \phi) \mid R    \rangle$ of $G$ relative to $\{A,B\}$, 
 the presentation complex is $C'(2\lambda)$ small cancellation complex.
 \item If $\lambda \leq  1/12$ then any compact generating graph of $G$ relative to $\{A,B\}$, then the Cayley-Abels graph is fine and hyperbolic. In particular, $G$ is relatively hyperbolic with respect to $\{A,B\}$. 
 \end{enumerate}
\end{proposition}
\begin{corollary}\label{exp:AmalSmallCancellation}
Let $G$ be as in Proposition~\ref{lem:AmalgamationHyperbolic} and $R = \{r^m\}$, where $r$ is a reduced and weakly reduced element in $G$. If $m \geq 12$ and $R$ satisfies the $C'(1/12)$ small cancellation condition, then $G$ is relatively hyperbolic with respect to $\{A, B\}$.
\end{corollary}

The rest of the section is divided into two subsections containing the proofs of  Theorem~\ref{prop:RelHypimpliesFinitePres}  and  Proposition~\ref{lem:AmalgamationHyperbolic} respectively.

\subsection{Proof of Theorem~\ref{prop:RelHypimpliesFinitePres}}\label{subsec:RelHypimpliesCP}

The proof uses the following result of Bowditch. For a definition of     $\Omega_n(\Gamma)$ see Definition~\ref{def:OmeganGamma}.

\begin{lemma} \label{lem:hypCircuits} \cite[Proposition 3.1]{Bo12}
 Let $\Gamma$ be a hyperbolic graph with hyperbolicity constant $k$. Then there is a constant $n = n(k)$ such that $\Omega_n(\Gamma)$ is simply-connected. 
\end{lemma}

\begin{proof}[Proof of Theorem~\ref{prop:RelHypimpliesFinitePres}]
Let $\Gamma$ be a Cayley-Abels graph of $G$ relative to $\mathcal{H}$.
By Remark~\ref{rem:relhyp}, $\Gamma$ is hyperbolic and fine. Let $n$ be the constant given by Lemma~\ref{lem:hypCircuits}. Since $\Gamma$ is fine and $G$-cocompact, there are finitely many $G$-orbits of circuits of length at most $n$. Then $X=\Omega_n(\Gamma)$ is a discrete  simply-connected cocompact 2-dimensional $G$-complex with 1-skeleton $\Gamma$. By  Theorem~\ref{prop:TopCharComplex}, $G$ is compactly presented with respect to   $\mathcal{H}$. 
\end{proof}

\subsection{Proof of Proposition~\ref{lem:AmalgamationHyperbolic}}
The combinatorial Dehn filling function of a simply-connected complex was introduced by~\cite{Gromov} as a generalization of isoperimetric functions. 
We  use the following result to prove fineness of certain complexes.

\begin{proposition} \cite{MP15} \cite{SamLouisEdu} \label{thm:Deltafine}
Let $X$ be a cocompact simply-connected $G$-complex. Suppose that each edge of $X$ is attached to finitely many 2-cells. Then the 1-skeleton of $X$ is a fine graph if and only if the combinatorial Dehn function $\delta_X$ of $X$ takes only finite values. 
\end{proposition}

The if direction is proved in~\cite[Proposition 2.1]{MP15} for homological Dehn functions $FV_{X}(k)$. In~\cite[Section 3]{Ge96}, Gersten observed that   $FV_{X}(k) \preceq  \delta_X^1(k)$ and hence the if direction of the proposition above follows. The proof of the only if direction is the same argument as~\cite[Proof of Lemma 2.3]{SamLouisEdu}.  

\begin{proof}[Proof of Proposition~\ref{lem:AmalgamationHyperbolic}]
 Let $X$ be the  presentation complex of $G$ corresponding to the generalized presentation $\langle (A\ast_CB,\phi)\mid R \rangle$, see Example~\ref{exp:Amalgamation}. The proof of~\cite[Theorem 7.1]{ACCMP} shows that 
 that $X$ is a  $C'(2\lambda)$ complex under the assumption that $A$ and $B$ are profinite groups; the same argument works in our case without changes. 

For the second statement, assume that  $\lambda\leq 1/12$. By the first statement of the proposition, $X$ is a $C'(1/6)$ complex. By a classical result~\cite[Theorem 11.2]{LySc01},  the combinatorial Dehn function   $\delta_X(k)$ is linear and hence the 1-skeleton $X^{(1)}$ is hyperbolic by \cite[Theorem 2.3.D]{Gromov}. Further $X^{(1)}$ is fine by Proposition~\ref{thm:Deltafine}.  Therefore, $G$ is relatively hyperbolic with respect to $\{A,B\}$.
\end{proof}

\section{Coherence of Small cancellation products}
\label{sec:main}

The main result of this section is Theorem~\ref{thm:mainCoherence2} which is slightly more general than Theorem~\ref{thmX:main} in the introduction of the article.

\begin{definition}
Let $G$ be a topological group, $N$ a normal subgroup, and $\phi\colon G \to G/N$ the quotient homomorphism. A class  $\mathcal J$ of subgroups of $G$ is $N$-stable if $\{\phi(Q)\mid Q\in \mathcal{J}\}$ is closed under finite intersections. 
\end{definition}

\begin{example}
Let $G$ be a topological group and let $N$ be a compact normal subgroup. Open subgroups and closed subgroups are examples of $N$-stable classes of subgroups. % I don't know whether Borel sets are $N$-stable since the continuous image of a Borel set is not necessarily a Borel set.
\end{example}

Theorem~\ref{thmX:main} in the introduction is obtained by taking $\mathcal J$ as the class of closed subgroups in the  result below.

\begin{theorem}\label{thm:mainCoherence2}
    Let $A \ast_C B$ be a topological group that splits as an amalgamated free product of two  open subgroups $A$ and $B$ with compact intersection $C$. Let $r\in A\ast_CB$. There is an integer $M=M(A\ast_CB, r)>0$ with the following property.

  Let $m>1$ such that  $r^m$ satisfies the $C'(\lambda)$ small cancellation condition for $\lambda$ such that $12 \lambda M<1$. Let $G$ be the  quotient topological group $(A\ast_BC)/\nclose{r^m}$, $\phi\colon A\ast_CB\to G$ the quotient map, and  $\mathcal J$  a collection of closed subgroups of $G$ such that: \begin{enumerate} 
    \item $\mathcal J$ is closed under conjugation,
 \item $\mathcal J$ contains the (images in $G$ of the) subgroups $A$ and $B$.
 \item $\mathcal{J}$ is $N$-stable where $N$ is the compact normal subgroup $\bigcap_{g\in G}g\phi(C)g^{-1}$.  
 \end{enumerate} 
If $A$ and $B$ are $\mathcal J$-coherent, then $G$ is $\mathcal J$-coherent. 
\end{theorem}

 The proof of the theorem reduces to prove that $G/N$ is $\mathcal I$-coherent where $\mathcal I=\{\phi(Q)\mid Q\in \mathcal{J}\}$ by invoking the following lemma.  

 \begin{lemma}\label{rem:CoCompactcoherent2}
    Let $G$ be a topological group,  $N \trianglelefteq G$  a compact normal subgroup, $\phi\colon G \to G/N$ the quotient map,    $\mathcal J$  a class of closed subgroups of $G$, and  $\mathcal I=\{\phi(Q)\colon Q\in \mathcal{J}\}$. Then   $G$ is $\mathcal J$-coherent if and only if $G/N$ is $\mathcal I$-coherent.
\end{lemma}
\begin{proof}
For $Q\in \mathcal{J}$ consider the short exact sequence of continuous homomorphisms $1\to Q\cap N \to Q \to \phi(Q) \to 1$. Since $N$ is compact and $Q$ is closed, we have that $Q\cap N$ is compact. By Proposition~\ref{prop:ConShortExact}, we have that $Q$ is compactly generated (resp. compactly presented) if and only if $\phi(Q)$ is compactly generated (resp. compactly presented).
Therefore $G$ is $\mathcal{J}$-coherent if and only if $G/N$ is $\mathcal{I}$-coherent.
\end{proof}

\subsection{Consequences of McCammond and Wise's Perimeter method}
A  complex $X$  is  \emph{$M$-thin} if  for every 1-cell $e$ of $X$, the set
\[  \{D \mid D \text{ is a 2-cell in } X \text{ and   $e$ belongs to $\partial D$}\}    \]
has cardinality at most $M$. A complex $X$ is \emph{uniformly circumscribed} if there is $L$ such that for any 2-cell  of $X$, its   boundary cycle has length at most $L$.

\begin{proposition} \label{thm:EMPQuasi03}
Let $X$ be a $C'(\lambda)$ small cancellation complex that is simply connected, uniformly circumscribed, $M$-thin, and $6\lambda M < 1$.  Let $Q$ be a topological group that acts faithfully and cellularly on $X$.
\begin{enumerate}
    \item \cite[Theorem 3.3]{MARTINEZPEDROZA20112396} If $Q$ is finitely generated relative to a finite collection of 0-cell stabilizers. Then there is a connected and quasi-isometrically embedded $Q$-cocompact subcomplex of $X$.

    \item If $Q$ is  compactly generated and acts discretely on $X$ with compact 1-cell $Q$-stabilizers, then $Q$ is hyperbolic relative  to a finite collection $\mathcal{K}$ of 0-cell $Q$-stabilizers. In particular, $Q$ is compactly presented with respect to $\mathcal{K}$.
 
 \item \label{part3}  Let $\mathcal J$ be a class of closed subgroups of $Q$ such that
 \begin{enumerate}
 \item $\mathcal J$ is closed under  finite intersections,
 \item $\mathcal J$ contains all  $Q$-stabilizers of $0$-cells of $X$.
 \end{enumerate}
  If $Q$ acts
  discretely on $X$ such that 1-cell $Q$-stabilizers are compact, and 0-cell $Q$-stabilizers are $\mathcal J$-coherent, then $Q$ is $\mathcal{J}$-coherent.
\end{enumerate}
\end{proposition}
\begin{proof}
To prove the second statement, let $U$ be the $Q$-stabilizer of a 1-cell of $X$. Observe that $U$ is a compact open subgroup of $Q$. Since $Q$ is compactly generated,  $Q$ is finitely generated with respect to $U$. Hence $Q$ is finitely generated with respect to the $Q$-stabilizer of a 0-cell of $X$.   By the first statement of the proposition,  there exists a connected $Q$-cocompact 1-dimensional subcomplex $Y$ quasi-isometrically embedded in $X$. The Dehn function of the $C'(1/6)$ simply connected small cancellation complex $X$ is linear~\cite[Ch.5 Thm 4.4]{LySc01} and hence the 1-skeleton  $X^{(1)}$ is a hyperbolic graph by \cite[Theorem 2.3.D]{Gromov}. Further $X^{(1)}$ is fine graph  by Theorem~\ref{thm:Deltafine}. Therefore, the $Q$-complex $Y$ is also hyperbolic and fine.  Since $Q$ acts cocompactly,   discretely and with compact edge stabilizers on $Y$, by Proposition~\ref{prop:BassTheory}, there exists a finite collection $\mathcal{K}$ of representatives of conjugacy classes of 0-cells $Q$-stabilizers such that there exists a compact generating graph of $Q$ with respect to $\mathcal K$. Moreover  the corresponding Cayley-Abels graph is $Q$-isomorphic to $Y$. Observe that $(Q,\mathcal K)$ is a proper pair.
Thus $Q$ is relatively hyperbolic with respect to $\mathcal{K}$ and by Theorem~\ref{prop:RelHypimpliesFinitePres}, $Q$ is compactly presented with respect to $\mathcal{K}$.

To prove the third statement, let $H \in \mathcal J$ be a compactly generated   subgroup of $Q$. Since $H$ is closed, $H$-stabilizers of 1-cells of $X$ are compact. Therefore, the second statement of the proposition implies that  
$H$ is compactly presented with respect to a finite collection $\mathcal{K}$ of $H$-stabilizers of $0$-cells of $X$. 
Since $H\in\mathcal J$ and $Q$-stabilizers of 1-cells of $X$ are in $\mathcal J$,  every subgroup  in $\mathcal{K}$ is in $\mathcal J$. 
Since $H$ is compactly generated, Theorem~\ref{prop:FiniteRelPresGisCGimpliesHisCG} implies that every subgroup in $\mathcal{K}$ is compactly generated. Since $Q$-stabilizers of $0$-cells of $X$ are $\mathcal J$-coherent, it follows that every subgroup in $\mathcal{K}$ is compactly presented. By Theorem~\ref{prop:FinRelPresHisCPimpliesGisCP}, $H$ is compactly presented. 
\end{proof}

\subsection{Proof of Theorem~\ref{thm:mainCoherence2}}
 Let $\widetilde{G}$ and $G$  denote the groups $A\ast_CB$ and $(A\ast_CB)/\langle \langle r^m\rangle \rangle $ respectively, and  let $\mathcal{T}$ be the  Bass-Serre tree of $\widetilde{G}$. Without loss of generality, assume  that the normal form $x_0x_1x_2 \cdots x_\ell$  of $r$ as an element of $A\ast_CB$ is a cyclically reduced word; see Section~\ref{Sec:Osin1} for a definition of normal form.  Let  $|r|$ denote the length of the normal form, that is, $|r|=\ell$.

Fix a vertex $y \in \mathcal{T}$, let $\gamma$ be the unique fixed path from $y$ to $r^2y$, and let \begin{equation}\label{eq:ConstantM} M = k|r|, \quad\text{where }  k = \max \left\{[\widetilde{G}_t\colon \widetilde{G}_{\gamma}] \mid t \text{ is a 1-cell in the image of } \gamma \right\}.\end{equation}
Note that $k$ is a finite integer.  Indeed, the pointwise $\widetilde G$-stabilizer  $\widetilde{G}_{\gamma}$ of the path $\gamma$ is the intersection of the $\widetilde G$-stabilizers of edges of $\gamma$. Since $\gamma$ is a finite path, $\widetilde{G}_{\gamma}$ is an open subgroup of $\widetilde{G}_t$. Since  $\widetilde{G}_t$ is a compact group, the index $[\widetilde{G}_t\colon \widetilde{G}_{\gamma}]$ is finite. 

Suppose that $m>1$ and $r^m$ satisfies the $C'(\lambda)$ small cancellation condition and $12\lambda M<1$. Consider a compact presentation 
\begin{equation}\label{eq:presentation} \langle \  (  A\ast_CB, \phi) \mid  r^m \   \rangle 
\end{equation} 
of $G$ relative to $\{A,B\}$, where $\phi\colon A\ast_CB \to G$ is a continuous epimorphism induced by the inclusions $A\hookrightarrow G$ and $B\hookrightarrow G$.  
 Let $X$ be the Cayley-Abels complex associated to~\eqref{eq:presentation}, see  Example~\ref{exp:Amalgamation}.
\begin{proposition}\label{prop:M-thin}
The $G$-complex $X$ is $M$-thin.
\end{proposition}

The proof of Proposition~\ref{prop:M-thin} is postponed to   Subsection~\ref{subsec:Mthin}. Let us  complete the proof of the theorem  using this proposition. 

Since $r^m$ satisfies the $C'(\lambda)$ small cancellation condition and $12\lambda M<1$,   Proposition~\ref{lem:AmalgamationHyperbolic} implies that  $X$ is a $C'(2\lambda)$ small cancellation complex and its 1-skeleton  is fine and hyperbolic. Since $G$ acts cocompactly, $X$ is uniformly circumscribed.

Consider the continuous homomorphism $G\to  Aut(X)$, observe that $N=\bigcap_{g\in G}g\phi(C)g^{-1}$ is its kernel, and let $Q$ be its image. Consider the short exact sequence $1\to N \to G \to Q \to 1$.  Then $Q$ acts faithfully, discretely, cocompactly and with compact 1-cell stabilizers on $X$. 

Since $A$ and $B$ are $\mathcal J$-coherent, Lemma~\ref{rem:CoCompactcoherent2} implies that   $A/N$ and $B/N$ are $\mathcal I$-coherent. On the other hand, $Q$-stabilizers of 0-cells  of $X$ are  conjugates in $Q$ of the subgroups $A/N$ and $B/N$. Therefore   $X$ has $\mathcal I$-coherent $Q$-stabilizers of 0-cells.
Since $\mathcal J$ is $N$-stable,    $\mathcal I$ is closed under finite intersections. Hence the  $Q$-action on $X$ and the collection of subgroups $\mathcal I$ satisfy the hypotheses of 
Proposition~\ref{thm:EMPQuasi03}\eqref{part3} and therefore $Q$ is $\mathcal I$-coherent. By Lemma~\ref{rem:CoCompactcoherent2}, $G$ is $\mathcal J$-coherent.

\subsection{Proof of Proposition~\ref{prop:M-thin}}\label{subsec:Mthin}

 Let $e$ be a 1-cell in $X$. The thinness of $e$ is given by the cardinality of the following $G_e$-set,
\[\mathcal{S}_e= \{(e,K) \mid K \text{ is a 2-cell in } X \text{ and } e \in \partial K\}.\] 
We prove below that the cardinality of $\mathcal{S}_e$ is bounded by $M$. 

Recall that $\mathcal{T}$ denotes the   Bass-Serre tree of $A\ast_CB$, $y$ is a fixed vertex of $\mathcal T$, and $\rho \colon \mathcal{T} \to X^{(1)} $ is the quotient map. Observe that for a reduced and cyclically reduced word $w$ of length $k>0$, the edge-length of the unique embedded path in the tree $\mathcal T$ from $y$ to $w.y$ is $k$.

By definition,  $X$ has a 2-cell $D$ whose boundary path is the image by $\rho$ of the path in $\mathcal T$ from $y$ to $r^my$. Let $G^D$ and $G_D$ denote the setwise and  pointwise stabilizer of $D$ respectively.
 Consider the finite set
\[\mathcal{U} = \{(t,D) \mid \text{ t is a 1-cell in }\partial D\}\]
and observe that $\mathcal{U}$ is a $G^D$-set. Let $U$ be the set of $G^D$-orbits of elements of  $\mathcal{U}$, and let  $S$ be the set of $G_e$-orbits in $\mathcal{S}_e$.

 \begin{claim*}
 The cardinality of $U$ is bounded by $|r|$.
 \end{claim*}
 \begin{proof}
 Since $\phi(r)$ setwise stabilizes $\partial D$ and no pair of distinct  2-cells of $X$ have the same boundary, $\phi(r)$ is an element of $G^D$.  Note that $\phi(r)$ has order $m$ in $G$. Let $q$ be the unique path in $\mathcal{T}$ from $y$ to $ry$. Then $\rho(q)$ is a subpath of $\partial D$ of length $|r|$. Since translates of the path $\rho(q)$ by $\langle \phi(r)\rangle $ cover $\partial D$, the set of elements $(t,D)$ where $t$ is a 1-cell in $\rho(q)$ contains representatives of all $G^D$-orbits of elements of $\mathcal U$. Hence 
 $U$ has at most $|r|$ elements.    
 \end{proof}

\begin{claim*}  There is an injective map from $S$ to $U$.
\end{claim*} 
\begin{proof}
 Since there is a single $G$-orbit of 2-cells in $X$, for any $(e,K) \in \mathcal{S}_e$, there exists a $g_K \in G$ such that $g_K.K = D$. This defines a map $\psi \colon S \to U$ given by $(e,K) \mapsto (g_K.e,g_K.K)$. 
To prove the claim, we show that for any pair of elements of $\mathcal{S}_e$, 
they are in the same $G_e$-orbit  if and only if their images under $\psi$ are in the same $G^D$-orbit. 

Let $(e,K)$ and $(e,K')$ be elements of $S_e$. First suppose there is $f \in G_e$ such that $(e,K)= f.(e,K')$. Then for $h = g_Kfg_{K'}^{-1} \in G^D$, we have $(g_K.e,D) = h.(g_{K'}.e,D)$.  Conversely, suppose that there exists $h \in G^D$ such that $\psi(e,K) = h.\psi(e,K')$. That is  $(g_K.e, g_K.K) = h.(g_{K'}.e,g_{K'}.K')$. Hence for $f =g_K^{-1}hg_{K'}  \in G_e$, we have that $(e,K)= f.(e,K')$. 
\end{proof} 
 
\begin{claim*} If $(e,K) \in \mathcal{S}_e$ then $[G_e \colon G_K] \leq k$.
\end{claim*}
\begin{proof}
Recall that $\gamma$ is the unique path between $y$ and $r^2y$ in $\mathcal{T}$. Consider the path $p = \rho(\gamma)$ in $X^{(1)}$. Note that $p$ is a subpath of $\partial D$.
We claim that $G_p = G_D$ where $G_D$ is the pointwise stabilizer of $D$. Observe that $G_D \subseteq G_p$. Conversely, if $g \in G_p$, then $g.D$ is a 2-cell of $X$ such that $|\partial D \cap \partial (g.D)|= 2|r|=\frac{2}{m}|\partial D|$ and hence by the small cancellation condition on $X$, $D = g.D$ pointwise and thus $g \in G_D$.

Let $(e,K) \in \mathcal{S}_e$ be any pair. Since  there is a single orbit of 2-cells in $X$, there exists 
$g \in G$ such that $g.K = D$; and since $G^D$ translates of $p$ cover $\partial D$, $g$ can be chosen such that
$g.e$ is in the image of $p$. 
Let $\widetilde{g.e}$ be a 1-cell in $\gamma$ such that $\rho(\widetilde{g.e})=g.e$. Since $\phi$ restricted to the stabilizer of edges is an isomorphism, we have the following commutative diagram

\[\begin{tikzcd}[column sep=small]
\widetilde{G}_{\gamma} \arrow[hookrightarrow]{r}{}  \arrow[rightarrow]{d}{\phi}
& \widetilde{G}_{\widetilde{g.e}} \arrow{d}{\phi} \\
G_p \arrow[hookrightarrow]{r}{} & G_{g.e},
\end{tikzcd}\]
and therefore, \[[G_e\colon G_K]=[G_{g.e} \colon G_D]=[G_{g.e}\colon G_p]=[\widetilde{G}_{\widetilde{g.e}}\colon \widetilde{G}_{\gamma}]\leq k.\qedhere\]
\end{proof}

The number of $G_e$-orbits in $\mathcal{S}_e$ is bounded by $|r|$ as a consequence of the first two claims. On the other hand, each $G_e$-orbit in $\mathcal{S}_e$ is bounded by $k$ by the last claim.  
Therefore, the cardinality of $\mathcal{S}_e$ is bounded by $|r|k= M$ and this concludes the proof of the Proposition~\ref{prop:M-thin}. 

\bibliographystyle{plain}
\bibliography{xbib}

\end{document}